\documentclass{article}
\usepackage[utf8]{inputenc}

\usepackage[english]{babel}

\usepackage[utf8]{inputenc}
\usepackage[T1]{fontenc}

\usepackage{amssymb}
\usepackage{stmaryrd}
\usepackage{amsmath,amsfonts}
\usepackage[tikz]{bclogo}
\usepackage[top=3cm,bottom=3cm,left=2.5cm,right=2.5cm]{geometry}

\usepackage{tikz}
\usepackage{tkz-tab}
\usepackage{caption}
\usepackage{latexsym}
\usepackage{amssymb}
\usepackage{amsmath}
\usepackage{subfig}
\usetikzlibrary{shapes.geometric}
\usetikzlibrary{decorations.pathreplacing}

\usepackage{lmodern}
\usepackage{textcomp}
\usepackage{mathtools, bm}
\usepackage{amssymb, bm}
\usepackage[unq]{unique}
\usepackage{enumerate}
\usepackage{comment}
\usepackage[breaklinks]{hyperref}

\usepackage[numbers]{natbib}  

\usepackage[breaklinks]{hyperref}
\usepackage[numbers]{natbib}
\usepackage{graphicx} 

\usepackage{amsthm}
\usepackage{multirow}

\usepackage{cleveref}

\newtheorem{theorem}{Theorem}[section]
\newtheorem{definition}[theorem]{Definition}
\newtheorem{lemma}[theorem]{Lemma}

\newtheorem{claim}[theorem]{Claim}

\newtheorem{Corollary}[theorem]{Corollary}

\newtheorem{observation}[theorem]{Observation}

\newcommand{\eps}{\varepsilon}
\newcommand{\Geom}{Geom}

\newcommand{\Po}{Po}
\newcommand{\R}{\mathbb{R}}

\DeclareMathOperator{\ind}{\mathbf{1}}

\title{Reconstructing a giant component of a point set in $\R$}
\author{Julien Portier \footnote{\href{mailto:jp899@cam.ac.uk}{julien.portier@epfl.ch}, Ecole Polytechnique Federale de Lausanne (EPFL), CH-1015 Lausanne,
Switzerland}}
\date{}

\begin{document}

\maketitle

\vspace{-1em}

\begin{abstract}
Let $V \subset \R$ be a finite set with $|V| = n $ and suppose we are given each pairwise distance independently with probability $p$. We show that if $p = (1+\eps)/n$, for some fixed $\eps >0$, then we can reconstruct a subset of size $\Omega_{\eps}(n)$, up to translation and reflection, with high probability. 
This confirms a conjecture posed by Gir\~ao, Illingworth, Michel, Powierski, and Scott.

We also study a deterministic variant proposed by Benjamini and Tzalik. 
We show that if we are given $m$ distinct pairwise distances of a point set $V \subset \R$ with $|V|=n$, then we can reconstruct a subset of size $\Omega(m/ (n \log n)) $, up to translation and reflection. 
Moreover, we show that this is optimal, which also disproves a conjecture posed by Benjamini and Tzalik.
\end{abstract}

\section{Introduction}
Benjamini and Tzalik \cite{benjamini2022determining} recently investigated the following question in the context of graph theory and distance reconstruction. 
Consider a set $V$ of $n$ distinct points in $\mathbb{R}$, where the only information available about $V$ consists of the pairwise distances between some of the pairs of points.
The central problem is: how much of the structure of $V$ can be deduced from this incomplete distance information? We first emphasize that we are concerned here with arbitrary point sets, and do not restrict ourselves to generic point sets. The two settings are fundamentally different, and we elaborate on these distinctions in the final paragraph of this introduction. \\

Our first main result addresses the probabilistic setting, when the pairwise distances are revealed randomly. 
This result confirms a conjecture by Gir{\~a}o, Illingworth, Michel, Powierski, and Scott.

\begin{theorem}
\label{thm:mainSparseRegime}
For $\eps >0$, let $V \subset \R$ with $|V| = n$. Suppose we are given each pairwise distance independently with probability $p= (1+\eps)/n$. Then we can reconstruct a subset of size $\Omega(\eps^3n)$ up to translation and reflection, with high probability as $n\rightarrow \infty$.
\end{theorem}

This result gives the sharp threshold for the reconstruction of a ``giant component'' and is sharp up to the power of $\eps$. Indeed one can show that except for some particular instances of $X$, any reconstructible set must be a subset of the $2$-core of $G(n,p)$, which has $\Theta( \eps^2 n) $ vertices when $p = (1+\eps)/n$.
In fact, we believe that a constant proportion of the $2$-core can be reconstructed and thus conjecture that the size of the reconstructible set in our first result can be improved from $\Omega(\eps^3 n)$ to $\Omega(\eps^2 n)$.

This paper comes on the back of several recent results in this area. Benjamini and Tzalik \cite{benjamini2022determining} showed that if $p = \Omega(\log n /n) $ then one can reconstruct all of $V$, with high probability. This result was then improved by 
Gir{\~a}o, Illingworth, Michel, Powierski, and Scott \cite{girao2023reconstructing} who showed that the same holds all the way to the threshold for every vertex having degree at least $2$, that is $np = \log n+\log \log n+\omega(1)$. This last result was then strengthened by Montgomery, Nenadov, Szabó, and the author \cite{montgomery2024global}, who showed that, one can in fact ``universally'' reconstruct the entirety of $V$ in this regime, for any point set $V \subset \R$ of size $n$, a notion which we will return to later.

The topic of sparser random graphs was also studied by  Gir{\~a}o, Illingworth, Michel, Powierski, and Scott, who  investigated the threshold that guarantees the existence of a reconstructible subset of linear size. Towards this they showed that one can take $p = 42/n$ in the setting of our theorem and conjectured that the constant $42$ could be reduced to $1+\eps$ with any fixed $\eps>0$. The first topic of this paper is to prove \Cref{thm:mainSparseRegime}, which confirms this conjecture. \\

Benjamini and Tzalik \cite{benjamini2022determining} also initiated the study of reconstruction in the deterministic case, where the distances are revealed according to an arbitrary graph.
To state our second main result, we need to introduce the notion of global rigidity in $\mathbb{R}$. 
We say that a graph $G=(V,E)$ is {\em globally rigid in $\mathbb{R}$} if for every injective functions $f,g :V\rightarrow \mathbb{R}$ satisfying $|f(x)-f(y)| = |g(x)-g(y)|$ for every edge $xy\in E$, we also have $|f(x)-f(y)| = |g(x)-g(y)|$ for every pair $x,y \in V$. 
In other words, a graph $G$ is globally rigid in $\mathbb{R}$ if, for any embedding of its vertices in $\mathbb{R}$, the distances along its edges uniquely determine the embedding up to isometry.
We establish the following result. 

\begin{theorem}
\label{thm:weakBT}
    Let $G$ be a graph on $n \geq 2$ vertices and $m$ edges.
    Then $G$ contains a subgraph on at least $c\frac{m}{n \log n}$ vertices that is globally rigid in $\mathbb{R}$, where $c=4 \log 2$.
\end{theorem}

For dense graphs, we can obtain slightly better bounds.

\begin{theorem}
    \label{thm:weakBTdense}
    Fix $\delta > 0$.
    Let $G$ be a graph on $n \geq 2$ vertices and $m \geq n^{1+\delta}$ edges.
    Then $G$ contains a subgraph on $\Omega_{\delta}(\frac{m}{n})$ vertices that is globally rigid in $\mathbb{R}$.
\end{theorem}

Benjamini and Tzalik \cite{benjamini2022determining} conjectured that for any point set $V \subset \R$ of size $n$, and any graph $G$ with vertex set $V$ and $m$ edges, if we reveal the pairwise distances of $V$ according to the edges of $G$, then there exists a subset of $V$ of size $m/n$ which is reconstructible.
They verified their conjecture for the case where $m \geq n^{3/2}$.
While \Cref{thm:weakBT} and \Cref{thm:weakBTdense} respectively show that Benjamini and Tzalik's conjecture is true up to an additional $\log n$ correcting factor on the size of the reconstructed set, and for dense enough graphs, we demonstrate that this conjecture does not hold in general through the following counterexample, which also shows that \Cref{thm:weakBT} and \Cref{thm:weakBTdense} are essentially best possible.

\begin{theorem}
\label{thm:CounterexampleBT}
    There exists a set $V \subset \R$ with $|V| =2^N$, and a graph $G=(V,E)$ such that $|E|=N2^{N-1}$, and such that no reconstructible subset of $V$ has size greater than $2$.
\end{theorem}

Our construction is essentially a well-chosen embedding of a hypercube graph; a structure that has previously appeared in the study of graph rigidity, see for instance Theorem 1.2 by Raz and Solymosi \cite{raz2023dense}.

\paragraph{Related notions of rigidity.}
Prior to the work of Benjamini and Tzalik \cite{benjamini2022determining} on arbitrary point sets, substantial research had already been conducted in the case of \emph{generic} point sets. 
A set $V \subset \mathbb{R}^d$ is called \emph{generic} if the $d|V|$ coordinates of the vertices are algebraically independent over $\mathbb{Q}$. 
An embedding $f \colon V \rightarrow \mathbb{R}^d$ is said to be \emph{generic} if its image is generic.
A graph $G=(V,E)$ is \emph{generic globally rigid in $\mathbb{R}^d$} if there is some generic embedding of $V$ in $\mathbb{R}^d$ which is reconstructible from the lengths along the edges of $G$. 
It has been established \cite{connelly2005generic, gortler_characterizing_2012} that a graph is generic globally rigid in $\mathbb{R}^d$ if and only if \emph{all} its generic embeddings are reconstructible from their edge lengths.
It is worth emphasizing, however, that the restriction to generic embeddings in the definition of global rigidity is a significant weakening. 
For instance, it is a classical result (see Theorem 63.2.7 in \cite{jordan2017global}) that a graph is generic globally rigid in $\mathbb{R}$ if and only if it is $2$-connected.
This stands in contrast to the reconstruction of arbitrary point sets, since Gir\~ao, Illingworth, Michel, Powierski, and Scott showed \cite{girao2023reconstructing} that there are graphs with arbitrarily high connectivity which can be embedded in $\mathbb{R}$ in such a way that their vertex sets cannot be reconstructed from their edge lengths.
In fact, those graphs not only reconstruct any generic embedding in $\mathbb{R}$, but also in $\mathbb{R}^d$ for arbitrarily large $d$, since, in a recent breakthrough, Vill\'anyi~\cite{villanyi2025every} showed that graphs of connectivity $d(d+1)$ reconstruct any generic point set in $\mathbb{R}^d$.

To draw a parallel with our main result \Cref{thm:mainSparseRegime}, it is also instructive to note developments in the study of generic global rigidity of random graphs. 
Namely, Lew, Nevo, Peled and Raz \cite{lew23randomrigid} proved that, for $d$ fixed, the random graph \(G(n,p)\) is generic globally \(d\)-rigid with high probability exactly when it has minimum degree \(d+1\), which has a sharp threshold at $p = (\log{n}+ d\log{\log{n}})/n$.
\\

The notion of rigidity is naturally connected to that of global rigidity. 
A graph $G=(V,E)$ is called rigid in $\R^d$, or $d$-rigid, if for a generic embedding
$p : V \rightarrow \R^d$, every continuous motion of the vertices in $\R^d$ that starts at $p$, and preserves the lengths of all the edges of $G$, does not change the distance between
any two vertices.
The rigidity of random graphs has been the subject of several recent papers, see for instance \cite{krivelevich2023rigid,lew23randomrigid,peled2024rigidity}. \\

The paper is structured as follows.
In \Cref{sec:outline}, we present a gentle outline of the proofs.
In Sections \ref{sec:Geometric-Lemmas} to \ref{sec:MainSparseRegime}, we introduce the tools required for our proofs.
In \Cref{sec:MainProof}, we prove \Cref{thm:mainSparseRegime}.
In \Cref{sec:Deterministic}, we prove Theorems \ref{thm:weakBT}, \ref{thm:weakBTdense} and \ref{thm:CounterexampleBT}.
Finally, we conclude with some remarks in \Cref{sec:ConcludingRemarks}.

\section{Outline of the proofs}
\label{sec:outline}

We start by giving an outline of the proof of \Cref{thm:mainSparseRegime}.
Let $G$ be a graph sampled as $G(n,(1+\eps)/n)$.
Our first tool is \Cref{lem:PropertyWitness} which gives us some structural property on $G$ whenever some distance $uv$ is not reconstructible from the distances on the edges of $G$. This is an adapted version of a result of Garamv{\"o}lgyi (see \Cref{thm:GaramvolgyiCriteria}), and provides a structural "witness" to the failure of reconstruction.
The intuition of this result is as follows: if the distance $uv$ is not reconstructible, then there exists two embeddings $f$ and $g$ of $V$ into $\R$ which agree on the edges of $G$ but differ on the distance $uv$.
For every edge $e=yz \in E(G)$, we say that $f$ and $g$ \emph{align} on $e$ if $f(y)-f(z)=g(y)-g(z)$, i.e. if the edge $e$ has same "direction" in both embeddings.
Potentially by replacing $g$ by its reflection $-g$, we may assume that $f$ and $g$ align on a set of edges $Z \subseteq E(G)$ of size $|Z| \geq |E(G)|/2$. 
Taking the $S_i$'s to be the connected components of $Z$, it is then easy to verify that the properties claimed hold: in particular that between two blocks $S_i$ and $S_j$, all edges of $G$ have same length relative to the embedding.
We call such a partition a witness, and, to summarise, \Cref{lem:PropertyWitness} states that the failure of reconstruction implies the existence of such a witness. \\

Our high-level strategy for the rest of the proof is then relatively simple: we aim to show that some subgraph $G'$ of $G$ of size at least $\Omega(\eps^3n)$ does not contain any witness, with high probability, which by the discussion of the previous paragraph, would show that any distances in $G'$ is reconstructible.
To show this, we expose first the \emph{skeleton} of $G$, i.e. the graph obtained from $G$ by removing the labels of the vertices, and select (see \Cref{sec:PrelSparseRegime}) a subgraph $G'$ of $G$ of size at least $\Omega(\eps^3n)$.
We fix a partition of the vertex set of $G'$, and we show that this partition forms a witness with small probability upon revealing the labels of the skeleton of $G'$.
A simple union bound over all such partitions would then give the desired result. \\

However, bounding the probability that a fixed partition is a witness is a complicated task itself.
For this, we exploit a key constraint imposed by \Cref{lem:PropertyWitness}: all edges between two distinct blocks $S_i$ and $S_j$ must have equal length in the embedding.
Therefore, we reveal the labels of the skeleton of $G'$ vertex by vertex, and once both endpoints of an edge connecting $S_i$ and $S_j$ are revealed, we have knowledge of the length of this edge, and all other edges connecting $S_i$ and $S_j$ must match it.
Since this event is highly unlikely when vertex labels are drawn uniformly, this gives a bound (formalized in \Cref{lem:Proba-witness}) on the probability that a given partition becomes a witness.
Note that this bound crucially depends on the number of edges crossing between blocks of the partition.
This motivates the choice of a subgraph $G'$ of $G$ having good expansion properties. 
However, previous results by Ding, Lubetzky, and Peres (\Cref{thm:DingAnatomy2coreSparseRegime}) and Benjamini, Kozma, and Wormald (\Cref{lem:CheegerSparseRegime}) imply together that the $2$-core of $G \sim G(n,(1+\eps)/n)$ is essentially an expander whose edges are subdivided according to some independent random variables.
Using those results, we will show that we can find some large subgraph $G'$ of $G$ with good expansion properties, for which we will be able to apply the above strategy. \\

Concerning the deterministic setting, we show \Cref{thm:CounterexampleBT} via a well-chosen embedding of the vertex set of the hypercube graph into $\mathbb{R}$.
Finally, we prove Theorems \ref{thm:weakBT} and \ref{thm:weakBTdense} using a characterization of globally rigid graphs in $\mathbb{R}$ by Garamv{\"o}lgyi, see \Cref{thm:GaramvolgyiCriteria}.

\section{Geometric lemmas}
\label{sec:Geometric-Lemmas}

We start this section by giving definitions of the terms used in the introduction and throughout the paper.
For convenience, we will often represent a point set $V$ of size $n$ as the image of an injective function $f \colon [n] \rightarrow \mathbb{R}$. 
Given an injective function $f \colon [n] \rightarrow \mathbb{R}$ in $\mathbb{R}$ and a graph $G$ on vertex set $[n]$, we define the {\em $G$-distance function of $f$}, denoted $d_{f,G}=d$, by $d(ij):= |f(i)-f(j)|$ for every $ij \in E(G)$.
We say that $f \colon [n] \rightarrow \mathbb{R}$ {\em realises} a function $d \colon E(G) \rightarrow \mathbb{R}$ if $d= d_{f, G}$. \\
We formally define distance reconstructibility as follows: Given an injective function $f \colon [n] \rightarrow \mathbb{R}$ in $\mathbb{R}$, a graph $G$ on vertex set $[n]$ and two vertices $u,v$ of $G$, for the sake of conciseness we say that \emph{the distance $uv$ is not reconstructible from $G$} if there exists an injective function $g \colon [n] \rightarrow \mathbb{R}$ such that $g$ realises $d_{f,G}$ but $|f(u)-f(v)| \neq |g(u)-g(v)|$. 
Otherwise, we say that \emph{the distance $uv$ is reconstructible from $G$}.
A subset $V' \subset V(G)$ is said to be \emph{reconstructible from $G$} if the distance $uv$ is reconstructible for any $u,v \in V'$. \\

We start with the following definition that is one of the main ingredients of our proofs.
A partition $S_1 \cup \dots \cup S_k$ of the vertex set $V(G)$ of a graph $G$ is said to be a \emph{connected partition} if the subgraph $G[S_i]$ is connected for each $i \in [k]$.

\begin{definition}
\label{def:witness}
    We fix an injective function $f \colon [n] \rightarrow \mathbb{R}$, and a graph $G$ on vertex set $[n]$.
We say that a connected partition $S_1 \cup \dots \cup S_k$ of $[n]$ is a \emph{witness} with respect to $G$ and $f$ if there exists some reals $f_{ij}$ for $i,j \in [k]$ with $f_{\ell \ell}=0$ for every $\ell \in [k]$ such that each of the following holds
\begin{itemize}
    \item the set of edges $W$ with endpoints in two different $S_i$ satisfies $|W| \leq E(G)/2$,
    \item for every $a_i \in S_i$ and $a_j \in S_j$ such that $i \neq j$ and $a_ia_j \in E(G)$, we have 
        \begin{align*}
            f(a_i)-f(a_j)=f_{ij},
        \end{align*}
    \item for any cycle $C=x_1 \dots x_r$ with $x_r=x_1$ in $G$, such that $x_i \in S_{y_i}$ for every $i \in [r]$, we have 
        \begin{align*}
            f_{y_1y_2}+\dots+f_{y_{r-1}y_r}=0.
        \end{align*}
\end{itemize}
\end{definition}

Before presenting our first result, we briefly discuss a related theorem due to Garamvölgyi, which provides useful context.
A 2-colouring of the edges of a graph $G$ is called a \emph{NAC-colouring} if both colours are used and no cycle has exactly one edge of a given colour. We shall always refer to the two colours as red and blue.
Using this notion, Garamv{\"o}lgyi showed in Theorem 2.4 of \cite{garamvolgyi2022global} the following characterisation of globally rigid graphs in $\mathbb{R}$.
\begin{theorem}
\label{thm:GaramvolgyiCriteria}
A connected graph $G$ is globally rigid in $\mathbb{R}$ if and only if $G$ has no NAC-colouring for which $|R_i \cap B_j| \leq 1$ for every $1 \leq i \leq k$ and $1 \leq j \leq l$ where $R_1, \dots, R_k$ are $B_1, \dots, B_l$ are the vertex sets of the connected components of the subgraph of red and blue edges, respectively. 
\end{theorem}

We now introduce our first lemma, which resembles \Cref{thm:GaramvolgyiCriteria}.

\begin{lemma}
\label{lem:PropertyWitness}
    Let $f \colon [n] \rightarrow \mathbb{R}$ be an injective function, and $G$ be a connected graph on vertex set $[n]$.
    Let $u$ and $v$ be two vertices of $G$.
    Suppose that the distance $uv$ is not reconstructible from $G$.
    Then there exists a connected partition $S_1 \cup \dots \cup S_k$ of $[n]$ which is a witness with respect to $G$ and $f$, and such that $u \in S_1$ and $v \in S_2$.
\end{lemma}

\begin{proof}
    If the distance $uv$ is not reconstructible, then by definition there exists an embedding $g$ such that $g$ agrees with $f$ on the edges of $G$, but $|f(u)-f(v)| \neq |g(u)-g(v)|$.
    Let $W \subset E(G)$ be defined by the set of edges $e=xy$ such that $f(x)-f(y)=g(y)-g(x)$. 
    Clearly $W$ is neither empty nor equal to $E(G)$, as otherwise $f$ and $g$ would be isometric.
    Additionally, by potentially replacing $g$ with $-g$, we ensure that $|W| \leq E(G)/2$.
    Let $S_1 \cup \dots \cup S_k$ be the partition of connected components of $G \setminus W$.
    Note that $u$ and $v$ belong to different connected components, as otherwise there would exist vertices $w_1, \dots, w_k$ with $w_1=u$ and $w_k=v$ such that $w_iw_{i+1} \in E(G) \setminus W$ for every $i \in [k-1]$, and therefore $f(w_i)-f(w_{i+1})=g(w_i)-g(w_{i+1})$ for every $i \in [k-1]$. Summing those would give $f(u)-f(v) = g(u)-g(v)$, contradicting $|f(u)-f(v)| \neq |g(u)-g(v)|$.
    Without loss of generality, we can relabel the components so that $u \in S_1$ and $v \in S_2$. \\
    We now show that $S_1 \cup \dots \cup S_k$ is a witness with respect to $G$ and $f$.
    To prove this, we show that there exists some reals $f_{ij}$ for $i,j \in [k]$ with $f_{\ell \ell}=0$ for every $\ell \in [k]$ satisfying the three conditions in \Cref{def:witness}. 
    Note that we have already shown that $|W| \leq E(G)/2$. \\
    We now move on to the second point. 
    Let $a_i, b_i \in S_i$ and $a_j, b_j \in S_j$ such that $i \neq j$, $a_ia_j \in E(G)$ and $b_ib_j \in E(G)$.
    Note that $a_ia_j \in W$ and $b_ib_j \in W$, as $i \neq j$.
    By summing along a path within $S_i$ from $a_i$ to $b_i$, we obtain $f(a_i)-f(b_i)=g(a_i)-g(b_i)$.
    Similarly, we have $f(a_j)-f(b_j)=g(a_j)-g(b_j)$.
    Therefore, we have
    \begin{align}
    \label{eq:Eqfai-faj1}
        f(a_i)-f(a_j) &= f(a_i)-f(b_i)+f(b_i)-f(b_j)+f(b_j)-f(a_j),
    \end{align}
    and
    \begin{align}
    \label{eq:Eqfai-faj2}
        f(a_i)-f(a_j) = g(a_j)-g(a_i) &= g(a_j)-g(b_j)+g(b_j)-g(b_i)+g(b_i)-g(a_i) \nonumber \\
        &= f(a_j)-f(b_j)+f(b_i)-f(b_j)+f(b_i)-f(a_i).
    \end{align}
    Adding \eqref{eq:Eqfai-faj1} and \eqref{eq:Eqfai-faj2} and simplifying gives $f(a_i)-f(a_j)=f(b_i)-f(b_j)$. 
    Therefore, setting $f_{\ell k} =  f(a_{\ell})-f(a_{k})$ for any $a_{\ell}a_k \in E(G)$ with  $a_{\ell} \in S_{\ell}$ and $a_{k} \in S_{k}$ gives the desired conclusion.  \\
    The third item in \Cref{def:witness} can then be proved the same way.
\end{proof}

For a permutation $\sigma$ of $[n]$ and an injective function $f \colon [n] \rightarrow$, we let $f_{\sigma}$ be the composition $f \circ \sigma$. 

We fix an injective function $f \colon [n] \rightarrow \mathbb{R}$, a connected graph $G$ on vertex set $[\ell]$ and a connected partition $S_1 \cup \dots \cup S_k$  of $[\ell]$, with $\ell \leq n/2$. 
Let $E'$ be the set of edges in $G$ with endpoints in two different $S_i$, $V'$ be the number of vertices which are incident to an edge in $E'$, and $C_2$ is the number of non-trivial connected components (those with more than one vertex) in the subgraph spanned by $E'$. 

The following result bounds the probability of existence of a witness for a given graph and a composition of the initial embedding with a permutation taken uniformly at random.

\begin{lemma}
\label{lem:Proba-witness}
    For a permutation $\sigma$ of $[n]$ taken uniformly at random, the probability $q$ of $S_1 \cup \dots \cup S_k$ being a witness with respect to $G$ and $f_{\sigma}$ satisfies
    \begin{align*}
        q \leq \left (\frac{n}{2} \right)^{-V'+C_2+k-1}.
    \end{align*}
\end{lemma}

\begin{proof}
    Let $K_1, \dots, K_{C_2}$ be the connected components of the subgraph spanned by $E'$, and for every $i \in [C_2]$, we let $T_i$ be an arbitrary spanning tree of $C_i$. 
    Let $F= \cup_{i=1}^{C_2} T_i$, and note that $F$ is a spanning forest of the subgraph spanned by $E'$, and has exactly $V'-C_2$ edges.
    Let $R$ be the empty graph on vertex set $[k]$, and let $Z$ be an initially empty set.
    We will sequentially reveal the images of the vertices in $F$ through $\sigma$, one by one, through a Depth-First Search (DFS), starting in each component of $F$ at an arbitrary vertex.
    Suppose we are currently exploring the edge $e=uv$ of $F$, and that we have already revealed the image of $u$.
    If $u$ is in $S_i$ and $v$ is in $S_j$, we proceed as follows: if $i$ and $j$ are in different components in $R$, then we add the edge $ij$ in $R$, otherwise, we add $e$ to $Z$.
    For each $e=uv \in Z$, before revealing the image of $v$, we know from the second item of \Cref{def:witness} that $f(\sigma(u))-f(\sigma(v))$ must be of a specific value to avoid ruling out $S_1 \cup \dots \cup S_k$ as a witness.
    Given that the images of at most $\ell \leq n/2$ elements have already been revealed and that the embedding $f$ is injective, the probability of this condition being satisfied is at most $2/n$, independently for each edge.
    Since $|Z| \geq V'-C_2-(k-1)$, the result follows. 
\end{proof} 

\section{Anatomy of the giant component in the supercritical regime and consequences}

Ding, Lubetzky, and Peres established \cite{ding2014anatomy} a useful description of the giant component of $G(n,p)$ for $p=\frac{\lambda}{n}$ where $\lambda > 1$ is a fixed constant.
We shall use the following consequence of their work concerning the $2$-core of the giant component of $G(n,p)$ for this range of $p$.
We denote by $\Po(\Lambda)$ the Poisson distribution with parameter $\Lambda$.

\begin{definition}
\label{def:Ding-Lubetzky-Peres-model}
    Let $\lambda >1$ be fixed.
    Let $\mu < 1$ be the conjugate of $\lambda$, that is $\mu e^{-\mu}=\lambda e^{-\lambda}$.
    We define the model $\tilde{\mathcal{C}}^{(2)}_{\lambda}$ the following way:
    \begin{enumerate}
    \item Let $\Lambda$ be Gaussian $N(\lambda-\mu, \frac{1}{n})$ and let $D_u \sim \Po(\Lambda)$ for $u \in [n]$ be iid, conditioned on $\sum D_u \ind_{D_u \geq 3}$ being even. Let $N_k= \sum \ind_{D_u=k}$ and $N=\sum_{k \geq 3} N_k$. Select a random multigraph $\mathcal{K}$ on $N$ vertices, uniformly among all multigraphs with $N_k$ vertices of degree $k$ for $k \geq 3$.
    \item Replace the edges of $\mathcal{K}$ by paths of iid $\Geom(1- \mu)$ lengths.
\end{enumerate}
\end{definition}

\begin{theorem}
\label{thm:DingAnatomy2coreSparseRegime}
Let $\lambda >1$ be fixed.
Let $\mathcal{C}^{(2)}=\mathcal{C}^{(2)}_{\lambda}$ be the $2$-core of the giant component of $G(n,p)$ for $p=\frac{\lambda}{n}$.
Then $\mathcal{C}^{(2)}$ is contiguous to the model $\tilde{\mathcal{C}}^{(2)}=\tilde{\mathcal{C}}^{(2)}_{\lambda}$ as defined in \Cref{def:Ding-Lubetzky-Peres-model}.
That is $\mathbb{P}(\tilde{\mathcal{C}}^{(2)} \in \mathcal{A}) \rightarrow 0$ implies $\mathbb{P}(\mathcal{C}^{(2)} \in \mathcal{A}) \rightarrow 0$ for any event $\mathcal{A}$.
\end{theorem}

In a graph $G$, a \emph{bare path} is a path $u_0u_1\dots u_k$ for some $k \in \mathbb{N}$ such that the internal vertices $u_1, \dots, u_{k-1}$ all have degree $2$ in $G$.
The \emph{kernel} $\mathcal{K}(G)$ of a graph $G$ is the graph $G'$ obtained from $G$ by replacing every maximal bare path with a single edge between their endpoints.

\begin{lemma}
\label{lem:BoundModelAnatomySparseRegime}
    Let $\eps \in (0,1)$ be fixed, and let $\lambda=1+\eps$. 
    Let $\mathcal{K}$ be the kernel of the model $\tilde{\mathcal{C}}^{(2)}_1$ defined in \Cref{def:Ding-Lubetzky-Peres-model}.
    Then for small enough $\eps$ we have $|V(\mathcal{K})| \in [\eps^3 n/1000,16\eps^3 n]$ and $|E(\mathcal{K})| \in [\eps^3 n/1000, 32 \eps^3 n]$, with high probability.
    Moreover, $\mathcal{K}$ has maximum degree at most $10 \log n$ with high probability.
\end{lemma}

\begin{proof}
    As $1-\eps \leq \mu \leq 1-\eps/2$, it follows that $\Lambda \in [\eps/3;2\eps]$, with high probability, by standard estimates on Gaussian random variables.
    Conditionally on the event $\Lambda \in [\eps/3;2\eps]$, we have for a given vertex $u \in [n]$ and small enough $\eps$:
    \begin{align*}
        \mathbb{P}(D_u \geq 3) = 1-e^{-\Lambda}-e^{-\Lambda}\Lambda -e^{-\Lambda}\Lambda^2/2 \in [ \Lambda^3/10;\Lambda^3].
    \end{align*}
    Therefore an application of Chernoff bound shows that $\sum_u \ind_{D_u \geq 3} \in [\Lambda^3n/20;2\Lambda^3n]$ with high probability.
    Moreover, it is easy to show that $\mathbb{P}(\sum D_u \ind_{D_u \geq 3} \text{ is even}) \geq 1/10$ for $n$ sufficiently large.
    Therefore $N \in [\eps^3 n/1000,16\eps^3 n]$ with high probability, as wanted.
    We treat the case of the edges in a similar way. We have for a given vertex $u \in [n]$ and small enough $\eps$:
    \begin{align*}
        \mathbb{E}[D_u \ind_{D_u \geq 3}] = \sum_{k \geq 3} e^{-\Lambda}\frac{\Lambda^k}{k!}k = \Lambda(1-e^{-\Lambda}-e^{-\Lambda}\Lambda) \in [\Lambda^3/4;2\Lambda^3].
    \end{align*}
    Another application of Chernoff bound shows that whp the number of edges in the kernel is in the interval $[\Lambda^3 n/8, 4 \Lambda^3 n]$, as wanted. \\
    Concerning the bound on the maximum degree of $\mathcal{K}$, note that by Corollary 2.2 of \cite{janson2011random}, which can be applied for a Poisson distribution by Remark 2.6 of \cite{janson2011random}, we have
    \begin{align*}
        \mathbb{P}(D_u \geq 10 \log n) &\leq \mathbb{P}( D_u \geq 5\log n+ \Lambda) \leq \exp\left( -\frac{25\log^2 n}{2(\Lambda+5\log n)}\right) = o(n^{-1}).
    \end{align*}
    A union bound gives the desired conclusion.
\end{proof}

For a graph $G$ and a subset $S \subseteq V(G)$, we write 
\begin{align*}
    \Phi(S)=\frac{e(S,S^c)}{d(S)},
\end{align*}
where $d(S)=\sum_{v \in S} d(v)$. 
We say that a graph $G$ is an \emph{$\alpha$-expander} if for every $S \subseteq V(G)$ with $|S| \leq |V(G)|/2$, we have $\Phi(S) \geq \alpha$.
Bollob{\'a}s showed \cite{bollobas1988isoperimetric} that with high probability random $r$-regular graphs are $c$-expander for some absolute constant $c>0$.
In Lemma~5.3 of \cite{benjamini2014mixing}, Benjamini, Kozma, and Wormald later proved a more general expansion result for random multigraphs with sparse prescribed
degree sequences.
The following consequence adapts their result to our
definition of expansion.

\begin{lemma}
\label{lem:CheegerSparseRegime}
    For every fixed constant $A \geq 3$, there exists an absolute constant $\gamma >0$ such that the following holds. 
    Let $n \in \mathbb{N}$, and let $d_1, \dots, d_n$ be a sequence of integers such that each of them is in the interval $[3, n^{0.02}]$ and $\sum_{i=1}^n d_i \leq An$ is even.
    Let $G$ be a random multigraph with given degree distribution $d_1, \dots, d_n$.
    Then $G$ is a $\gamma$-expander, with high probability.
\end{lemma} 

\begin{proof}
    By Lemma~5.3 of \cite{benjamini2014mixing}, there exists an absolute constant $h>0$ such that, with high probability, for every $S \subseteq V(G)$, we have
\[
    e_G(S,S^c)
    \geq h \min\{d_G(S),d_G(S^c)\}.
\]
We work on this event.
Let $S \subseteq V(G)$ satisfy $1 \leq |S|\leq n/2$.
Since $G$ has minimum degree at least three, we have
\[
    d_G(S^c) \geq 3|S^c| \geq \frac{3n}{2}.
\]
Moreover, the bound on the total degree gives
\[
    d_G(S)
    = d_G(V(G)) - d_G(S^c)
    \leq An - \frac{3n}{2}.
\]
Consequently,
\[
    \frac{d_G(S^c)}{d_G(S)}
    \geq \frac{3}{2A-3} .
\]
As $A\geq 3$, we have $3/(2A-3)\leq 1$, and hence
\[
    \frac{e_G(S,S^c)}{d_G(S)}
    \geq
    h \min\left\{1,\frac{d_G(S^c)}{d_G(S)}\right\}
    \geq \frac{3h}{2A-3}.
\]
Thus $G$ is a $\gamma$-expander, with
$\gamma = 3h/(2A-3)$.
\end{proof}

A direct corollary of \Cref{thm:DingAnatomy2coreSparseRegime}, the second point of Lemma \ref{lem:BoundModelAnatomySparseRegime}, and Lemma \ref{lem:CheegerSparseRegime} is the following.

\begin{Corollary}
\label{cor:2coreExpanderSparseRegime}
There exists an absolute constant $\alpha>0$ such that, for every fixed $\eps>0$ and $p = (1+\eps)/n$, the kernel of the $2$-core of $G(n,p)$ is an $\alpha$-expander, with high probability. 
\end{Corollary}

\begin{proof}
    By \Cref{lem:BoundModelAnatomySparseRegime}, with high probability, for every $u\in[n]$,
    we have $D_u\leq 10\log n$ in \Cref{def:Ding-Lubetzky-Peres-model}.
    Let $\mathcal{K}$ be the kernel of the model
    $\tilde{\mathcal{C}}^{(2)}_{\lambda}$ from \Cref{def:Ding-Lubetzky-Peres-model},
    where $\lambda=1+\varepsilon$, and let
    $N=|V(\mathcal{K})|$.
    Still by \Cref{lem:BoundModelAnatomySparseRegime}, with high probability, we have
    \[
        N\geq \frac{\varepsilon^3 n}{1000}
        \qquad\text{and}\qquad
        |E(\mathcal{K})|\leq 32\varepsilon^3 n.
    \]
    In particular,
    \[
        \sum_{v\in V(\mathcal{K})} d_{\mathcal{K}}(v)
        =2|E(\mathcal{K})|
        \leq 10^5N.
    \]
    Moreover, $\mathcal{K}$ has minimum degree at least three,
    and $10\log n\leq N^{0.02}$ for sufficiently large $n$.
    Therefore, conditional on its degree sequence,
    Lemma~\ref{lem:CheegerSparseRegime} applies with $A=10^5$,
    and shows that $\mathcal{K}$ is an $\alpha$-expander
    with high probability, for some absolute constant
    $\alpha>0$.
    As the $2$-core of the giant component of $G(n,p)$
    is contiguous to the model
    $\tilde{\mathcal{C}}^{(2)}_{\lambda}$ by \Cref{thm:DingAnatomy2coreSparseRegime},
    it follows that its kernel is an $\alpha$-expander
    with high probability.
\end{proof}

We will also make use of the following lemma bounding the number of maximal bare paths in the $2$-core of $G(n,\frac{1+\eps}{n})$. 
Recall that in a graph $G$, a \emph{bare path} is a path $u_0u_1\dots u_k$ for some $k \in \mathbb{N}$ such that the internal vertices $u_1, \dots, u_{k-1}$ all have degree $2$ in $G$.
Also recall that the \emph{kernel} $\mathcal{K}(G)$ of a graph $G$ is the graph $G'$ obtained from $G$ by replacing every maximal bare path with a single edge between their endpoints.

\begin{lemma}
\label{lem:NumberBarePathsSparseRegime}
For $\eps >0$ and $0 < \gamma < 1$ fixed, let $p = (1+\eps)/n$, and let $G \sim G(n,p)$. 
Then the $2$-core of $G$ contains at most $(\gamma/10^8)|\mathcal{K}(G)|$ maximal bare paths of length at least $100/(\gamma \eps)$, with high probability. 
\end{lemma}

\begin{proof}
    By \Cref{thm:DingAnatomy2coreSparseRegime}, it suffices to prove this statement for the model $\tilde{\mathcal{C}}^{(2)}_1$ described in that result.
    Let $\tilde{\mathcal{K}}$ be the kernel of the $2$-core in this model.
    By Lemma \ref{lem:BoundModelAnatomySparseRegime} we have $|V(\tilde{\mathcal{K}})| = \Omega(\eps^3 n)$, with high probability.
    In this model, each edge between two vertices of $\tilde{\mathcal{K}}$ is subdivided by a random variable following the geometric distribution $\Geom(1-\mu)$, where each subdivision is independent of all the others.
    Given that $\mu \leq 1-\eps/2$, we can bound the probability that the length of such a path is at least $100 /(\gamma \eps)$ as follows:
    \begin{align*}
        \mathbb{P}\left(\Geom(1-\mu) \geq \frac{100}{\gamma \eps}\right) = \sum_{k \geq 100 /(\gamma \eps)} \mu^{k-1} (1-\mu) &= \mu^{100/(\gamma \eps)} \leq (1-\eps/2)^{100/(\gamma \eps)} \leq \gamma/10^9.
    \end{align*}
    A standard application of Chernoff bound concludes the proof. 
\end{proof}

\section{Resilience of expanders}

In this section we show the following lemma about resilience of expanders with minimum degree at least $3$ with respect to removing a small enough number of edges.

\begin{lemma}
\label{lem:RemovingEdges2core}
     For $c\in (0,1]$, let $m, N$ be integers satisfying $1\leq m \leq cN/100$. 
     Let $G$ be a $c$-expander graph on $N$ vertices, with $\delta(G) \geq 3$, and let $E_0 \subset E(G)$, with $|E_0| = m$. Then there exists a subgraph $G'$ of $G \setminus E_0$ such that
    \begin{itemize}
        \item $|V(G')| \geq N/2$;
         \item $\delta(G')\geq 2$;
        \item $G'$ contains at least $N/4$ vertices of degree at least $3$;
        \item $G'$ is a $c/10$-expander.
    \end{itemize}
\end{lemma}

\begin{proof}
    We will define a sequence of subgraphs $(G_i)_{i \geq 0}$ of $G$ such that $G_{i+1} \subseteq G_i$, and a (non-negative) weight function $w_i$ on the edge set of $G$ for each $i$.
    We greedily remove vertices and edges from the graph $G_i$ while the expansion condition or the minimum degree condition is violated. 
    By definition we will have $\sum w_i(e) \geq 0$, but each of these modifications will decrease the value of $\sum w_i(e)$ by at least $4c|S_i|/5$, where $S_i$ is the set of vertices removed from $G_i$ to obtain $G_{i+1}$.
    This will ensure the algorithm terminates without having removed many vertices.
    We now step through the details of this proof. \\
    We first define a weight function $w : E(G) \rightarrow \R_{\geq 0}$ on the edges of the graph $G$ as follows: each edge in the set $E_0$ initially receives a weight of 2, while all other edges are assigned a weight of 0.
    Next we set $G_0=G \setminus E_0$ and $w_0=w$.
    We define the graph $G_{i+1}$ from $G_i$ and update the weight function $w_{i+1}$ from $w_i$ so that for every deleted edge $e$, the weight $w_i(e)$ is simply the number of its endpoints that are present in the graph $G_i$, whereas for every non-deleted edge $e$, we set $w_i(e)$ to $0$.
    More precisely, we proceed as follows:
    \begin{itemize}
        \item If $G_i$ contains a vertex $u$ of degree $0$, or a vertex $u$ of degree 1 (with a single neighbour $v$), then we remove vertex $u$ and, (if it exists) the edge $uv$ from $G_i$, obtaining $G_{i+1}$.
        Let $v_1, \dots, v_k$ be the other neighbours of $u$ in $G$, and note that $k \geq 2$ as we assumed that $G$ has minimum degree at least $3$.
        We update the weight function to obtain $w_{i+1}$ from $w_i$ by decreasing the weight of each edge $uv_1, \dots, uv_k$ by $1$, and setting the weight $w_{i+1}(uv)=1$, if the vertex $v$ exists.
        Note that this ensures 
        \begin{align} \label{eq:IneqWeight1SparseRegime}
            1+\sum_{e \in E(G)} w_{i+1}(e) \leq \sum_{e \in E(G)} w_{i}(e).
        \end{align}
        \item If $G_i$ contains a set of vertices $S$ such that $e_{G_i}(S,S^c) < c(\sum_{v \in S} d_G(v))/10$ and $|S| \leq |V(G_i)|/2$, then we remove all vertices of $S$ from $G_i$ to obtain $G_{i+1}$.
        The weight function is updated so that each edge in $E(G_i)$ with exactly one endpoint in $S$ is set to 1, and for each edge in $E(G) \setminus E(G_i)$ with exactly one endpoint in $S$, the weight is decreased by 1. The weight of any edge in $E(S)$ is set to $0$.
        As $G$ is by assumption a $c$-expander, this ensures that 
        \begin{align} \label{eq:IneqWeight2SparseRegime}
            4c|S|/5+\sum_{e \in E(G)} w_{i+1}(e) \leq 4c\sum_{v \in S} d_G(v)/5+\sum_{e \in E(G)} w_{i+1}(e) \leq \sum_{e \in E(G)} w_{i}(e).
        \end{align}
    \end{itemize}
    We claim that the resulting graph $G'$ from this process has the desired properties.
    By construction, $G'$ is a subgraph of $G \setminus E_m$, has minimum degree $2$ and is a $c/10$-expander.
    Hence, the only remaining step is to verify that $G'$ has size at least $N/2$ and contains at least $N/4$ vertices of degree at least $3$.
    Let $D$ be the total number of vertices removed during the process.
    Let $S_i$ and $S_f$ be the sum of the weights at the beginning and end of the process, respectively. 
    By definition, $S_i = 2m$ and since the weight of an edge is non-negative, we have $S_f \geq 0$. By \eqref{eq:IneqWeight1SparseRegime} and \eqref{eq:IneqWeight2SparseRegime} we have $S_f \leq S_i - 4cD/5$.
    By combining the three previous equations, we obtain $D \leq \frac{10m}{4c} \leq N/2$.
    Moreover, let $M$ be the number of vertices of degree $2$ in $G'$. 
    It follows that $S_f \geq M/2$, and therefore $M/2 \leq S_i \leq N/50$.
    Therefore, $G'$ contains at least $N/4$ vertices of degree at least $3$, as wanted.
\end{proof}

\section{Setting up the host graph for reconstruction} 
\label{sec:PrelSparseRegime}

We start by introducing the following definition.
A graph $H$ is said to be a $(n,\eps, \gamma)$-\emph{good graph} if the following holds:
\begin{enumerate}
        \item $H$ contains at least $\eps^3 n/10^5$ vertices of degree at least $3$,
        \item $H$ is a $\gamma \eps$-expander,
        \item $H$ contains no cycle of length at most $\frac{\log n}{2\eps}$,
        \item $H$ contains no bare path of length at least $\frac{1}{\gamma \eps}$,
        \item $\Delta(H) \leq 100 \log \log n$ and $\delta(H) \geq 2$,
        \item if $S \subseteq V(H)$ then $e(S,S^c) \geq 2$, where $S^c=V(H) \setminus S$
        \item if $S \subseteq V(H)$, $|S| \leq |V(H)|/2$ and $S$ contains at least one vertex of degree at least $3$, then $|N_{H}(S)|\geq 3$.
    \end{enumerate}

Using the tools developed in the previous parts of this paper, we show in this section that we can find a $(n,\eps,\gamma)$-good graph with high probability in $G(n,(1+\eps)/n)$, for some absolute constant $\gamma>0$.

\begin{lemma}
\label{lem:Structure2coreSparseRegime}
    There exists some constant $\gamma>0$ such that the following holds.
    For $\eps >0$ fixed and sufficiently small, let $p = (1+\eps)/n$ and let $G \sim G(n,p)$. 
    Then with high probability there exists a subgraph $G'$ of $G$ such that $G'$ is $(n,\eps,\gamma)$-good.
\end{lemma}

We will first need the following results about random graphs. 
The first result is a simple consequence of Chernoff inequality.

\begin{lemma}
\label{lem:HighDegreeRandomGraphsSparseRegime}
For fixed $\eps \in (0,1)$, let $p = (1+\eps)/n$.  
Then $G(n,p)$ contains at most $n e^{-10 \log \log n}$ vertices of degree at least $100 \log \log n$, with high probability. 
\end{lemma}

\begin{proof}
    For each set $S$ of size $|S| =n e^{- 10 \log \log n}$, if each element of $S$ has degree at least $100 \log \log n$ in $G$, then the number of edges in $G[S] \cup G[S, S^c]$ is at least $n e^{-10\log \log n} 50 \log \log n$. 
    Thus, the probability of this event can be bounded from above by 
    \begin{align*}
        \mathbb{P}(Bin(2n|S|,p) \geq 50|S|\log \log n).
    \end{align*}
    By Chernoff's bound, this probability is at most $e^{-20|S| \log \log n}$.
    A simple union bound over all sets $S$ then shows that this bad event happens with probability at most
    \begin{align*}
        \binom{n}{|S|} e^{-20|S|\log \log n} \leq \left(\frac{en}{|S|}\right)^{|S|} e^{-20|S|\log \log n} \leq e^{10|S| \log \log n+O(|S|)-20|S| \log \log n} =o(1).
    \end{align*}
    This finishes the proof.
\end{proof}

\begin{lemma}
\label{lem:NumberShortCyclesSparseRegime}
For fixed $\eps>0$, let $p = (1+\eps)/n$. 
Then $G(n,p)$ contains at most $n^{3/4}$ cycles of length at most $\frac{\log n}{2 \eps}$, with high probability.
\end{lemma}

\begin{proof}
Let $G$ be sampled as $G(n,\frac{1+\eps}{n})$.
    Let $C_k$ be the number of cycles of size $k$ in $G$, and let $L_{\frac{\log n}{2\eps}}$ be the number of cycles of size at most $\frac{\log n}{2\eps}$ in $G$. 
    We have, for any $k \leq \frac{\log n}{2\eps}$, that
    \begin{align*}
        \mathbb{E}[C_k] = \binom{n}{k} \frac{(k-1)!p^k}{2} \leq e^{\eps k} \leq n^{1/2}.
    \end{align*}
    Therefore $\mathbb{E}[L_{\frac{\log n}{2\eps}}] \leq n^{2/3}$ and Markov's inequality concludes.
\end{proof}

We are now ready to prove \Cref{lem:Structure2coreSparseRegime}.

\begin{proof}[Proof of \Cref{lem:Structure2coreSparseRegime}]
    Let $\mathcal{C}$ be the $2$-core of $G$, and let $\mathcal{K}$ be the kernel of $\mathcal{C}$.
    By Corollary \ref{cor:2coreExpanderSparseRegime}, there exists an absolute constant $\alpha > 0$ such that $\mathcal{K}$ is an $\alpha$-expander, with high probability.
    Let $\gamma = \alpha^2/3000$.
    Let $E_m$ be the set of edges of $\mathcal{K}$ corresponding to bare paths of length at least $100/(\alpha \eps)$ (in $\mathcal{C}$), edges incident to a vertex of degree at least $100 \log \log n$ (in $\mathcal{C}$), and one edge picked in each cycle of length at most $\frac{\log n}{2 \eps}$ (in $\mathcal{C}$).
    By Lemmas \ref{lem:NumberBarePathsSparseRegime}, \ref{lem:BoundModelAnatomySparseRegime}, \ref{lem:HighDegreeRandomGraphsSparseRegime}, and \ref{lem:NumberShortCyclesSparseRegime}, we may apply Lemma \ref{lem:RemovingEdges2core} to $\mathcal{K}$ and $E_m$ to obtain $\mathcal{K}'$.
    Let $G'$ be the graph obtained from $\mathcal{C}$ by removing the associated edges.
    It is clear that $G'$ satisfies the conditions $1$, $3$ and $5$. 
    
    Moreover, since $\mathcal{K}'$ is an $\alpha/10$-expander and $G'$ can be obtained from $\mathcal{K}'$ by subdividing each edge by at most $100/(\alpha \eps)$ times, then $G'$ is a $\alpha^2 \eps/3000$-expander. As $\gamma = \alpha^2/3000$, this shows condition $2$. 
    
    As $\mathcal{K}'$ is a $\alpha/10$-expander, it has in particular no bare path of length $20/\alpha$. 
    As again $G'$ can be obtained from $\mathcal{K}'$ by subdividing each edge by at most $100/(\alpha \eps)$ times, it follows that $G'$ does not contain any bare path of length at least $\frac{2000}{\alpha^2 \eps}$. As $\gamma = \alpha^2/3000$, condition 4 is satisfied. 
    
    Concerning condition $6$, it suffices to deal with the case $|S| \leq |V(G')|/2$.
    If $|S| \geq \log n/(2\eps)$, then we are done as $G'$ is a $\gamma \eps$-expander.
    On the other hand, if $|S| \leq \log n/(2\eps)$, then $S$ does not contain a cycle, and as $G'$ has minimum degree at least $2$, we are done in this case too. 
    
    We prove that condition $7$ holds with a similar reasoning. 
    If $|S| \geq \log n/(2\eps)-3$, then the result follows immediately, as $G'$ is a $\gamma \eps$-expander and has maximum degree at most $100 \log \log n$.
    On the other hand, if $|S| \leq \log n/(2\eps)-3$, then $S$ does not contain a cycle.
    Moreover, $S$ contains a vertex $v$ of degree at least $3$, say with neighbours $w^1$, $w^2$ and $w^3$.
    We construct a sequence $w^1_0, \dots, w^1_t$, starting with $w^1_0=w^1$, and for each $i$, we define $w^1_{i+1}$ to be a neighbour of $w^1_{i}$ in $S$, distinct from $w^1_0, \dots, w^1_i$.   
    When this process stops at step $t$, as $S$ contains no cycle, it follows that $w^1_t$ has a neighbour $y^1$ outside of $S$. 
    We repeat the same process for $w^2$ and $w^3$, and find vertices $y^2$ and $y^3$ in $S^c$, each of them having a neighbour in $S$.
    As $|S \cup \{ y^1, y^2, y^3 \}| \leq \log n/(2\eps)$, it follows that $S \cup \{ y^1, y^2, y^3 \}$ contains no cycle, and therefore $y^1$, $y^2$ and $y^3$ are distinct, finishing the proof.
\end{proof}

\section{The size of the summation set}

In this section, we essentially provide bounds on the size of the set over which the main union bound in the proof of \Cref{thm:mainSparseRegime} is taken. \\

Before we proceed to our result, we recall a few auxiliary tools. 
We begin with the following classical result, that can be found in \cite{bollobas2006art}, exercise 45. 

\begin{lemma}
\label{lem:Number-connected-subgraphsSparseRegime}
    Let $G$ be a graph on $n$ vertices and maximum degree $\Delta$.
    Then the number of connected induced subgraphs of $G$ of size $s$ is at most $n (e\Delta)^{s}$.
\end{lemma}

In addition, we require the following result by Chv{\'a}tal \cite{chvatal1991almost}.

\begin{theorem}
    Let $\eps \in (0,1.88)$ be fixed. Then the $3$-core of $G(n,(1+\eps)/n)$ is empty, with high probability.
\end{theorem}

We will also use the following folklore result.

\begin{theorem}
    Let $G$ be a graph of average degree $d$. Then $G$ has a subgraph with minimum degree at least $d/2$.
\end{theorem}

From the two previous results, we immediately deduce the following lemma.

\begin{lemma}
\label{lem:Small-average-degreeSparseRegime}
Let $\eps \in (0,1.88)$ be fixed. Then every subgraph of $G(n,(1+\eps)/n)$ has average degree at most $6$, with high probability.
\end{lemma}

Let $H$ be a graph, and let $S_1, \dots, S_k$ be a partition of $V(H)$. 
We define a \emph{super-block} the following way.
Suppose $S_k$ has the largest size amongst $S_1, \dots, S_k$. 
We construct the auxiliary graph $\mathcal{H}$ whose vertices are the blocks $S_i$ for $i \leq k-1$.
For distinct $i,j \in [k-1]$, the two blocks $S_i$ and $S_j$ are connected by an edge in $\mathcal{H}$ if $e_{H}(S_i,S_j) > 0$.
A \emph{super-block} is then defined as a connected component in this auxiliary graph $\mathcal{H}$. 
For integers $s_1,\dots, s_{k-1}$ and a graph $H$, we let $f(s_1,\dots,s_{k-1},H)$ be the number of connected partitions $S_1 \cup \dots \cup S_k$ of $V(H)$ such that for each $i \in [k-1]$, we have $|S_i|=s_i$, and the number of super-blocks is exactly 1. 
The main bound we establish in this section is the following.

\begin{lemma}
\label{lem:Bound-Connected-Partitions}
    Fix $\eps > 0$ sufficiently small, and let $G$ be sampled as $G(n,(1+\eps)/n)$. 
    Then with high probability, for every subgraph $G'$ of $G$ with maximum degree at most $100 \log \log n$, we have
    \begin{align*}
        f(s_1,\dots,s_{k-1}, G') \leq n(1000 \log \log n)^{2(s_1+\dots+s_{k-1})}.
    \end{align*}
\end{lemma}

\begin{proof}
    First, by \Cref{lem:Small-average-degreeSparseRegime}, every subgraph of $G$ has average degree at most $6$, with high probability, and we assume that this event holds for the rest of the proof.
    Let $S_1 \cup \dots \cup S_k$ of $V(H)$ be a connected partition such that for each $i \in [k-1]$, we have $|S_i|=s_i$, and the number of super-blocks is exactly 1.
    We now consider function $\phi$ which is defined the following way.
First, we take a spanning tree $T$ of $G'[\cup_{ i \in [k-1]} S_i ]$ by joining spanning trees of the $G'[S_i]$.
We then take an arbitrary root $r$ of this spanning tree. 
This root defines an ordering of the tree by DFS. 
For each $S_i$, let $r_i$ be the vertex in $S_i$ with the lowest index.
We now define the image of $\phi$ to be this spanning tree $T$, together with the root $r$, and the set $\{ r_i : i \in [k-1]\} $.
It is straightforward to check that the map $\phi$ defined this way is injective. 
We now obtain the desired bound on $f(s_1,\dots,s_{k-1}, G')$ via a simple counting argument.
For $\cup_{ i \in [k-1]} S_i $ we have $n (100 e \log \log n)^{s_1+\dots+s_{k-1}}$ choices by Lemma \ref{lem:Number-connected-subgraphsSparseRegime}.
As any subgraph of $G'$ has average degree at most $6$, we have at most $2^{3(s_1+\dots+s_{k-1})}$ choices of subgraphs with vertex set $\cup_{ i \in [k-1]} S_i $, so in particular at most $2^{3(s_1+\dots+s_{k-1})}$ spanning trees, and $s_1+\dots+s_{k-1}$ choices for the root.
Finally, we have $\binom{s_1+\dots+s_{k-1}}{k-1} \leq 2^{s_1+\dots+s_{k-1}}$ choices for the set $\{ r_i : i \in [k-1]\} $.
It follows that
\begin{align*}
    f(s_1,\dots,s_{k-1}, G') &\leq  n (100 e \log \log n)^{s_1+\dots+s_{k-1}} (s_1+\dots+s_{k-1})2^{4(s_1+\dots+s_{k-1})} \nonumber \\
    &\leq n(1000 \log \log n)^{2(s_1+\dots+s_{k-1})}.
\end{align*}
\end{proof}

\section{Technical estimates for the proof of Theorem \ref{thm:mainSparseRegime}}
\label{sec:MainSparseRegime}

As explained in the outline in \Cref{sec:outline}, our proof of \Cref{thm:mainSparseRegime} will consist of a first moment method on a $(n, \eps, \gamma)$-good subgraph $G'$ of $G$, whose existence is guaranteed with high probability by Lemma \ref{lem:Structure2coreSparseRegime}.
By \Cref{lem:PropertyWitness}, it suffices to show that $G'$ does not admit any witness with high probability.
We take a random uniform permutation $\sigma$ on $V$, and it therefore suffices to show that
\begin{align*}
        \sum_k \sum_{\cup_{i \leq k} S_i=V(G')}  \mathbb{P}_{\sigma}(S_1 \cup \dots \cup S_k \text{ is a witness of }G'_{\sigma}) =o_{\eps}(1),
    \end{align*} 
    where the second sum is over each connected partition, and $G'_{\sigma}$ is the graph obtained from $G'$ by relabelling its vertex set according to $\sigma$.
    We may now apply \Cref{lem:Proba-witness} to show the above inequality. 
    However, due to the nature of the bound on $\mathbb{P}_{\sigma}(S_1 \cup \dots \cup S_k \text{ is a witness})$ given by Lemma \ref{lem:Proba-witness}, this will require us to give efficient lower bounds on the quantity $V'-C_2-(k-1)$ appearing in Lemma \ref{lem:Proba-witness}.
    The goal of this section is to provide such estimates in the form of the next two lemmas.

In the following, we fix $\eps \in (0,1)$, some constant $\gamma$ as given in Lemma \ref{lem:Structure2coreSparseRegime}, and we let $c_{\eps}=\gamma \eps$.
We fix a $(n,\eps, \gamma)$-good graph $G'$, and let $N$ be the number of vertices of $G'$.
We further fix a connected partition $S_1, \dots, S_k$ of $V(G')$ such that $S_1, \dots, S_{k-1}$ is a super-block.
For every $i \in [k-1]$, we let $s_i=|S_i|$.
Let $E'$ be the set of edges having endpoints in two different $S_i$, and suppose that the set $E'$ is spanning over $V'$ vertices, forming a set $\mathcal{C}_2$ of $C_2$ non-trivial connected components (those with more than one vertex). 
Finally, we set $C_1=k-1$.
Our two main technical estimates are the following.

\begin{lemma}
    \label{lem:IneqSuperblock}
    Suppose that $s_1+\dots +s_{k-1} \leq N/2$, and that at least one of $S_1, \dots, S_{k-1}$ contains a vertex of degree at least $3$. 
    Then the following hold: 
        \begin{itemize}
            \item $V'-C_2-C_1 \geq \frac{c_{\eps}(s_1+\dots+s_{k-1})}{400\log \log n}$,
            \item $V'-C_2-C_1 \geq 2$.
        \end{itemize}
    \end{lemma}

 \begin{lemma}
    \label{lem:CountingEdgesRevealedCase2}
        There exists a constant $\eta =\eta(\eps) >0$ such that the following holds for any sufficiently small $\eps$. 
        Suppose that $\sum_{i=1}^k E(G'[S_i]) \geq N/2$, and that $s_1+\dots+s_{k-1} \geq N/2$.
        Then 
        \begin{align*}
            V'-C_1-C_2 \geq \frac{\eta N}{(\log \log N)^2}.
        \end{align*}
    \end{lemma}

We now proceed to the proofs of \Cref{lem:IneqSuperblock} and \Cref{lem:CountingEdgesRevealedCase2}.
\begin{proof}[Proof of \Cref{lem:IneqSuperblock}]
    We start by proving the first inequality.
    Let $\mathcal{V}$ be the set of vertices that are the endpoints of at least one edge in $E'$, and note that $|\mathcal{V}|=V'$.
    Let $\tilde{S}_1$ be the set of $S_i$ for $i \in [k-1]$ that contains exactly one vertex in $\mathcal{V}$, and let $\tilde{S}_2$ be the set of $S_i$ for $i \in [k-1]$ with at least two vertices in $\mathcal{V}$.
    Note that if $S_i$ is in $\tilde{S}_1$, then the block $S_i$ is reduced to a single point. Indeed, if the block $S_i$ was not reduced to a single point, then $S_i$ contains a cycle, and thus $|S_i| \geq (\log n)/(2\eps)$. But then $e(S_i,S_i^c) \geq \gamma \eps (\log n)/(2\eps)$, and since the degree of any vertex is at most $100 \log \log n$, we obtain a contradiction.
    Let $V_1$ be the set of vertices of $\mathcal{V}$ that are part of a block in $\tilde{S}_1$, and let $V_2$ be the set of vertices of $\mathcal{V}$ that are part of a block in $\tilde{S}_2$.
    Define $V_{int}=\mathcal{V} \cap (\cup_{i \in [k-1]} S_i)$ and $V_{ext}=\mathcal{V} \cap S_k$. 
    Notice that $V'=|V_{int}|+|V_{ext}|$, that $|V_{int}| = |V_1|+|V_2|$ and that $C_1 \leq |V_1|+|V_2|/2$. Thus we have
    \begin{align}
    \label{eq:V'C1C2lem31}
        V'-C_1-C_2 \geq |V_{ext}|+ |V_2|/2-C_2.
    \end{align}
    We claim the following.
    \begin{claim}
        At least one of the following holds for every component $C$ in $\mathcal{C}_2$: 
    \begin{itemize}
        \item $C$ contains at least two vertices in $V_2$.
        \item $C$ contains at least one vertex in $V_2$ and at least one vertex in $V_{ext}$.
        \item $C$ contains at least two vertices in $V_{ext}$.
    \end{itemize}
    \end{claim}
    \begin{proof}
    To prove this claim, we suppose that $C$ contains exactly one vertex in $V_2 \cup V_{ext}$, denoted by $u$.
    Then $C$ must contain a cycle (as all elements of $C$, except $u$, are in $V_1$, and thus have degree at least $2$ in $C$).
    Thus $|C| \geq \frac{\log n}{2\eps}$, and therefore letting $C^{1}=C \cap V_1$, we have that $|C^{1}| \geq \frac{\log n}{2\eps}-1$, and consequently $e(C^{1},(C^{1})^c) \geq c_{\eps}\log n$ for $n$ large enough, giving a contradiction, as $E(C^{1},(C^{1})^c) \subseteq E'$, and thus $u$ would have degree at least $c_{\eps}\log n/2$.
    The case where $C$ contains no vertex in $V_2 \cup V_{ext}$ is treated the same way.
    \end{proof}
    Let $\gamma_1$, $\gamma_2$, and $\gamma_3$ be the number of each type above, so that $C_2=\gamma_1+\gamma_2+\gamma_3$. 
    We then have $|V_2| \geq 2\gamma_1+\gamma_2$, and $|V_{ext}| \geq \gamma_2+2\gamma_3$.
    Therefore, plugging this into \eqref{eq:V'C1C2lem31} gives 
    \begin{align}
    \label{eq:V'C1C2gammas}
        V'-C_1-C_2 \geq |V_{ext}|+ |V_2|/2-C_2 &\geq \gamma_2/2+\gamma_3 \\
        &\geq (\gamma_2+\gamma_3)/2. \nonumber 
    \end{align}
    Therefore, if $\gamma_2+\gamma_3 \geq \frac{c_{\eps}(s_1+\dots+s_{k-1})}{200\log \log n}$, then we are done. 
    If not, we simply replace the inequality $|V_{ext}| \geq \gamma_2+2\gamma_3$ by $|V_{ext}| \geq \frac{c_{\eps}(s_1+\dots+s_{k-1})}{100 \log \log n}$ which holds by Lemma \ref{lem:Structure2coreSparseRegime}, and get
    \begin{align*}
         V'-C_1-C_2 &\geq \frac{c_{\eps}(s_1+\dots+s_{k-1})}{100 \log \log n}- \gamma_2-\gamma_3 \geq \frac{c_{\eps}(s_1+\dots+s_{k-1})}{200 \log \log n},
    \end{align*}
    as desired. \\
     For the second inequality, note that $|V_{ext}| \geq 3$ by point 7) of Lemma \ref{lem:Structure2coreSparseRegime}. Therefore we have 
    \begin{align}
    \label{eq:V'C1C2gammascasedeg3}
        V'-C_1-C_2 \geq  |V_{ext}|+ |V_2|/2-C_2 \geq 3-\gamma_2/2-\gamma_3.
    \end{align}
    As $V'-C_1-C_2$ is an integer, adding \eqref{eq:V'C1C2gammas} and \eqref{eq:V'C1C2gammascasedeg3} gives $V'-C_1-C_2 \geq 2$, as wanted.
    \end{proof}

    \begin{proof}[Proof of \Cref{lem:CountingEdgesRevealedCase2}]
        Let $\mathcal{V}$ be the set of vertices that are the endpoints of at least one edge in $E'$, and note that $|\mathcal{V}|=V'$.
        For a set $S \subseteq V(G')$, the \emph{border} of $S$ is the set of vertices $v \in S$ such that $d(v) \neq d_S(v)$.
        Let $\mathcal{C}_1$ be the set of components of $G' \setminus E'$ and let $\mathcal{C}_2$ be the set of components of the graph induced by $E'$, such that $|\mathcal{C}_1|=C_1$ and $|\mathcal{C}_2|=C_2$.
        For each component in $\mathcal{C}_1$, we choose a vertex from its border to be its representative.
        If possible, we choose a vertex incident to at least two edges in $E'$, if not, we choose this vertex in an arbitrary manner.
        The \emph{interior} of a bare path consists of all the vertices of degree exactly $2$ of this bare path.
        The components in $\mathcal{C}_1$ are classified as follows:
        \begin{itemize}
            \item Type $0$: A component consisting of a single vertex.
            \item Type $I$: A non-singleton component contained in the interior of a bare path of $G'$.
            \item Type $II$: A component with at least three boundary vertices, all of which having degree 1 in the graph induced by $E'$.
            \item Type $III$: Any component that does not fit into the previous categories.
        \end{itemize}
        Let $V_{rep}$ denote the set of representative vertices, and let $V_{ext}=\mathcal{V} \setminus V_{rep}$.
        As $V'-C_1-C_2= |V_{ext}| - C_2$, to prove the desired result, it suffices to show that "most" vertices of $V_{ext}$ have another vertex of $V_{ext}$ in their connected component in $\mathcal{C}_2$.
        We prove this the following way.
        Let $R_0$, $R_{I}$, $R_{II}$, and $R_{III}$ be respectively the set of components of type $0$, $I$, $II$, and $III$ in $\mathcal{C}_1$.
        Let $V_{ext,I}$, $V_{ext,II}$ and $V_{ext,III}$ be the set of vertices from $V_{ext}$ that are in a component of type $I$, $II$ and $III$, respectively.
        Similarly, let $V_{rep,0}$, $V_{rep,I}$, $V_{rep,II}$, and $V_{rep,III}$ be the corresponding subsets of $V_{rep}$.
        We repeatedly merge two adjacent components along
        a bare path of $G'$, provided that at least one
        of them is of type $I$ and the other is either
        of type $I$ or a singleton vertex of degree $2$.
        Note that this merging process does not change the value of the expression $V'-C_1-C_2$, that $s_1+\dots+s_{k-1} \geq N/2$, and that 
        \begin{align}
        \label{eq:BoundSumsi}
        \sum_{i=1}^k |E(G'[S_i])| \geq N/2.
        \end{align}
        From the outcome of this merging process we note the following observation.
        \begin{observation}
        \label{obs:Merging}
            If an edge $e=vw$ belongs to $E'$ and $v$ is on the border of a component belonging to $V_{I}$, then $w$ is not in $V_{I}$, and additionally, if $w$ is in a component belonging to $V_{0}$, then $w$ has degree at least $3$.
        \end{observation}
        We also make the following observation, which is straightforward from the definition of the type of a component in $\mathcal{C}_1$.
        \begin{observation}
        \label{obs:DegreeVrep}
            Let $v \in V_{rep}$, and suppose $v \in C$, where $C \in \mathcal{C}_2$.
            Then $d_C(v) \geq 2$ if $v \in V_{rep,0} \cup V_{rep,III}$, and $d_C(v) =1$ if $v \in V_{rep,I} \cup V_{rep,II}$.
        \end{observation}
        We now claim the following.
        \begin{claim}
        \label{cl:Classification1Vext}
            For every connected component $C$ in $\mathcal{C}_2$ which contains exactly one vertex $v_{ext}$ of $V_{ext}$, at least one of the following holds.
        \begin{itemize}
            \item $v_{ext}$ is of type $I$, and $C$ contains at most one edge between $v_{ext}$ and $V_{rep}$, and at least one vertex from $V_{rep,II}$.
            \item $v_{ext}$ is of type $I$, and $C$ contains at most one edge between $v_{ext}$ and $V_{rep}$, and at least one vertex from $V_{rep,0}$ and two vertices from $V_{rep,I}$.
            \item $v_{ext}$ is of type $I$, and $C$ contains at most one edge between $v_{ext}$ and $V_{rep}$, and at least one vertex from $V_{rep,III}$ and one vertex from $V_{rep,I}$.
            \item $v_{ext}$ is of type $I$, and $C$ contains at most one edge between $v_{ext}$ and $V_{rep}$, and at least $\log n$ vertices from $V_{rep,0} \cup V_{rep,III}$.
            \item $v_{ext}$ is of type $II$, and $C$ contains at most one edge between $v_{ext}$ and $V_{rep}$, and at least one vertex from $V_{rep,I}$.
            \item $v_{ext}$ is of type $II$, and $C$ contains at most one edge between $v_{ext}$ and $V_{rep}$, and at least one vertex from $V_{rep,II}$.
            \item $v_{ext}$ is of type $II$, and $C$ contains at most one edge between $v_{ext}$ and $V_{rep}$, and at least $\log n$ vertices from $V_{rep,0} \cup V_{rep,III}$.
            \item $v_{ext}$ is of type $III$, there is $e \geq 1$ edges between $v_{ext}$ and $V_{rep}$, and there exists $r$ such that $C$ contains $r$ vertices in $V_{rep,I} \cup V_{rep,II}$ and $(e-r) \log n$ vertices from $V_{rep,0} \cup V_{rep,III}$.
        \end{itemize}
        \end{claim}
        \begin{proof}
        Suppose that $v_{ext}$ is of type $I$. First, observe that by definition, there is at most one edge between $V_{ext}$ and $V_{rep}$ in $C$.
        We may also assume that $C$ contains no vertex from $V_{rep,II}$ (as otherwise outcome $1$ holds).
        Let $L$ be the number of leaves in $C$, and note that a leaf of $C$ is either $v_{ext}$, or belongs to $V_{rep,I}$ by \Cref{obs:DegreeVrep} (as we have excluded that $C$ contains a vertex from $V_{rep,II}$).
        If $L=1$, then $C$ contains a cycle, and therefore outcome $4$ holds.
        If $L \geq 2$, then there exists $u$ in $C \cap V_{rep,I}$.
        Based on \Cref{obs:Merging}, it follows that $C$ also contains a vertex in $V_{rep,III}$ or a vertex in $V_{rep,0}$ of degree at least $3$.
        In the first case, we are in outcome $3$, and in the second, either $L \geq 3$, leading to outcome $2$, or $C$ contains a cycle, which results in outcome $4$. \\
        The case where $v_{ext}$ is of type $II$ can be treated in a similar manner: again, by definition, there is at most one edge between $V_{ext}$ and $V_{rep}$ in $C$.
        Then we may suppose that $C$ contains no vertex from $V_{rep,II}$ (as otherwise outcome $6$ holds).
        Letting $L$ be the number of leaves in $C$, we have that a leaf of $C$ is either $v_{ext}$, or belongs to $V_{rep,I}$ by \Cref{obs:DegreeVrep}.
        If $L=1$, then $C$ contains a cycle, and therefore outcome $7$ holds.
        If $L \geq 2$, then there exists $u$ in $C \cap V_{rep,I}$, and outcome $5$ holds. \\
        The case where $v_{ext}$ is of type $III$ can also be treated in a similar manner: let $e$ be the number of neighbours of $v_{ext}$ in $V_{rep}$, call them $u_1, \dots, u_e$. 
        For every $i$, let $D_i$ be the set of vertices at distance at most $\log n$ from $u_i$ in the graph induced by $C \setminus v_{ext}$.
        Since $G'$ does not contain any cycle of length at most $2\log n +1$ for $\eps$ small enough, it follows that the sets $D_i$ are disjoint, and that each $D_i$ either contains at least one vertex in $V_{rep,I} \cup V_{rep,II}$, or $\log n$ vertices in $V_{rep,0} \cup V_{rep,III}$.
        Letting $r$ be the number of indices satisfying the first case, we obtain the desired result.
        \end{proof}
        Let $\alpha_1, \dots, \alpha_8$ be the number of components of each type in \Cref{cl:Classification1Vext}. 
        Additionally, let $\alpha_{4,0}$ and $\alpha_{4,III}$ be the number of vertices in outcome $4$ being respectively in $V_{rep,0}$ and $V_{rep,III}$, so that 
        \begin{align}
        \label{eq:BoundAlpha4Alpha40}
            \alpha_{4,0}+\alpha_{4,III} \geq \alpha_4 \log n.
        \end{align}
        Similarly, let $\alpha_{7,0}$ and $\alpha_{7,III}$ be the number of vertices in outcome $7$ being respectively in $V_{rep,0}$ and $V_{rep,III}$, so that 
        \begin{align}
        \label{eq:BoundAlpha7Alpha70}
            \alpha_{7,0}+\alpha_{7,III} \geq \alpha_7 \log n.
        \end{align}
        Let  $\beta_{8,0}$, $\beta'_{8,I}$, $\beta'_{8,II}$ and $\beta_{8,III}$ be respectively such that the number of vertices of type $0$, $I$, $II$ and $III$ in a component of outcome $8$ is respectively $\beta_{8,0} \log n$, $\beta'_{8,I}$, $\beta'_{8,II}$ and $\beta_{8,III} \log n$.
        Note that we have
        \begin{align}
        \label{ineq:beta8prime}
        \beta'_{8,I}+\beta'_{8,II} \geq \alpha_8-\beta_{8,0}-\beta_{8,III}.
        \end{align}
        We now claim the following.
        \begin{claim}
        \label{cl:ClassificationNoVext}
        For every connected component $C$ in $\mathcal{C}_2$ which contains no vertex from $V_{ext}$, at least one of the following holds. 
        \begin{itemize}
            \item $C$ contains at least $\log n$ vertices from $V_{rep,0} \cup V_{rep,III}$,
            \item $C$ contains at least three vertices from $V_{rep,I}$,
            \item $C$ contains at least two vertices from $V_{rep,I}$ and at least one vertex from $V_{rep, III}$,
            \item $C$ contains at least one vertex from $V_{rep,I}$, and at least one vertex from $V_{rep,II}$,
            \item $C$ contains at least two vertices from $V_{rep,II}$.
        \end{itemize}
        \end{claim}
        \begin{proof}
            Let $L$ be the number of leaves of $C$.
            If $L \leq 1$, then $C$ contains a cycle, and therefore outcome $1$ holds.
            If $L \geq 2$, then $C$ contains at least two vertices from $V_{rep,I} \cup V_{rep,II}$.
            Therefore, either $C$ contains at least two vertices from $V_{rep,II}$ (in which case outcome $5$ is satisfied), or $C$ contains at least one vertex from $V_{rep,I}$ and at least one vertex from $V_{rep,II}$ (in which case outcome $4$ is satisfied), or at least two vertices from $V_{rep,I}$.
            In that latter case, it follows from \Cref{obs:Merging} that $C$ either contains a vertex in $V_{rep,II} \cup V_{rep,III}$ (in which case outcome $3$ or $4$ holds) or a vertex in $V_{rep,0}$ of degree at least $3$.
            In the latter possibility, either $L \geq 3$ and outcome $2$ or $4$ holds, or $C$ contains a cycle and outcome $1$ holds.
        \end{proof}
        Let $\gamma_1$, $\gamma_2$, $\gamma_3$, $\gamma_4$ and $\gamma_5$ be the number of components of each of the types described in \Cref{cl:ClassificationNoVext}. \\
        Counting the number of representative vertices of type I, we get
        \begin{align}
        \label{ineq:VISparseRegime}
            |V_I| \geq 2 \alpha_2+\alpha_3+\alpha_5+\beta'_{8,I}+3\gamma_2+2\gamma_3+\gamma_4.
        \end{align}
        Counting the number of representative vertices of type II, we get
        \begin{align}
        \label{ineq:VIISparseRegime}
            |V_{II}| \geq \alpha_1+\alpha_6+\beta'_{8,II}+\gamma_4+2\gamma_5.
        \end{align}
        Counting the number of representative vertices of type III, we get
        \begin{align}
        \label{ineq:VIIISparseRegime}
            |V_{III}| \geq \alpha_3+ \alpha_{4,III} + \alpha_{7,III}+ \beta_{8,III} \log n +\gamma_3.
        \end{align}
        We say that a component in $\mathcal{C}_2$ is \emph{single} if it contains exactly one vertex in $V_{ext}$.
        The number of vertices $\delta_{I}$ from $V_{ext,I}$ which are not in a single component satisfies
        \begin{align}
        \label{ineq:LeftoverVextI}
            \delta_I \geq |V_{I}| - \alpha_1-\alpha_2-\alpha_3 -\alpha_4.
        \end{align}
        The number of vertices $\delta_{II}$ from $V_{ext,II}$ which are not in a single component satisfies
        \begin{align}
        \label{ineq:LeftoverVextII}
            \delta_{II} \geq 2|V_{II}| - \alpha_5-\alpha_6-\alpha_7.
        \end{align}
        The number of vertices $\delta_{III}$ from $V_{ext,III}$ which are not in a single component satisfies
        \begin{align}
        \label{ineq:LeftoverVextIII}
            \delta_{III} \geq |V_{III}| - \alpha_8.
        \end{align}
        The number of edges in between $V_{ext}$ and $V_{rep}$ which are part of a single component is
        \begin{align}
        \label{ineq:EdgesBetweenVextVrep}
            |E_{single}(V_{ext},V_{rep})| = \alpha_1+\dots+\alpha_7+\beta'_{8,I}+\beta'_{8,II}+\beta_{8,0}+\beta_{8,III}.
        \end{align}
        By simple rearrangement, we have
        \begin{align}
        \label{eq:SumNotAloneVext}
            2(V'-C_1-C_2) = 2(|V_{ext}|-C_2)
            &= 2(\sum_{C \in \mathcal{C}_2} |V_{ext} \cap C| - 1 ) \nonumber \\
            &\geq 2\left(\sum_{C \in \mathcal{C}_2, |V_{ext} \cap C| \geq 2} \frac{|V_{ext} \cap C|}{2} - \sum_{C \in \mathcal{C}_2, |V_{ext} \cap C| =0} 1 \right) \\
            &\geq 2\left(\frac{\delta_{I}+\delta_{II}+\delta_{III}}{2} -(\gamma_1+\gamma_2+\gamma_3+\gamma_4+\gamma_5) \right) \nonumber \\
            &\geq \delta_{I}+\frac{3}{4}\delta_{II}+\frac{1}{4}\delta_{III} -2(\gamma_1+\gamma_2+\gamma_3+\gamma_4+\gamma_5) . \nonumber
        \end{align}
        Plugging in \eqref{ineq:LeftoverVextI}, \eqref{ineq:LeftoverVextII} and \eqref{ineq:LeftoverVextIII}, we obtain
        \begin{align*}
            2(V'-C_1-C_2) \geq \,& |V_{I}| - \alpha_1-\alpha_2-\alpha_3 -\alpha_4+\frac{3}{4}(2|V_{II}| - \alpha_5-\alpha_6-\alpha_7)+ \frac{1}{4}(|V_{III}| - \alpha_8)  \\
            \,&-2\gamma_1-2\gamma_2-2\gamma_3-2\gamma_4 -2\gamma_5.
        \end{align*}
        Plugging in \eqref{ineq:VISparseRegime}, \eqref{ineq:VIISparseRegime} and \eqref{ineq:VIIISparseRegime}, we obtain
        \begin{align*}
            2(V'-C_1-C_2) \geq \,& \frac{1}{2}\alpha_1+\alpha_2+\frac{1}{4}\alpha_3+\frac{1}{4}\alpha_{4,III}-\alpha_4+ \frac{1}{4}\alpha_5+\frac{3}{4}\alpha_6+\frac{1}{4}\alpha_{7,III}-\frac{3}{4}\alpha_7 +\beta'_{8,I}\\
            &+\frac{3}{2}\beta'_{8,II}+\frac{1}{4}\beta_{8,III} \log n -\frac{1}{4}\alpha_8-2\gamma_1+\gamma_2+\frac{5}{4}\gamma_3+\frac{1}{2}\gamma_4+\gamma_5.
        \end{align*}
        Plugging in \eqref{ineq:beta8prime} multiplied by $1/4$, we get 
        \begin{align*}
            2(V'-C_1-C_2) \geq \,& \frac{1}{2}\alpha_1+\alpha_2+\frac{1}{4}\alpha_3+\frac{1}{4}\alpha_{4,III}-\alpha_4+ \frac{1}{4}\alpha_5+\frac{3}{4}\alpha_6+\frac{1}{4}\alpha_{7,III}-\frac{3}{4}\alpha_7+\frac{3}{4}\beta'_{8,I}+\frac{5}{4}\beta'_{8,II} \\
            &+\frac{1}{4}\beta_{8,III}- \frac{1}{4}\beta_{8,0}-2\gamma_1+\gamma_2+\frac{5}{4}\gamma_3+\frac{1}{2}\gamma_4+\gamma_5.
        \end{align*}
        Using \eqref{ineq:EdgesBetweenVextVrep}, it follows that
        \begin{align}
        \label{eq:BoundFinalManyBlocks}
            2(V'-C_1-C_2) \geq \frac{1}{4}|E_{single}(V_{ext},V_{rep})|+\frac{1}{2}(\gamma_1+\gamma_2+\gamma_3+\gamma_4+\gamma_5) -2(\alpha_4+\alpha_7+\beta_{8,0}+\gamma_1)
        \end{align}
        From \eqref{eq:BoundAlpha4Alpha40} and \eqref{eq:BoundAlpha7Alpha70}, it follows that $\alpha_4 =O(\frac{n}{\log n})$ and $\alpha_7 =O(\frac{n}{\log n})$.
        Furthermore, it is straightforward that $\gamma_1=O(\frac{n}{\log n})$ and $\beta_{8,0}=O(\frac{n}{\log n})$.
        Therefore, if we have
        \begin{align*}
            |E_{single}(V_{ext},V_{rep})| \geq \frac{n}{10^{10}(\log \log n)^2},
        \end{align*}
        then by \eqref{eq:BoundFinalManyBlocks}, we get the desired result, as $N \leq n$.
        Similarly, if we have
        \begin{align*}
            \gamma_1+\gamma_2+\gamma_3+\gamma_4 \geq \frac{c_{\eps} \eps^3 n}{10^{14} (\log \log n)^2},
        \end{align*}
        where $c_{\eps}=\gamma \eps$ for the constant $\gamma$ given in Lemma \ref{lem:Structure2coreSparseRegime}, then we are also done.
        Thus, we may assume that 
        \begin{itemize}
            \item $ |E_{single}(V_{ext},V_{rep})| \leq \frac{n}{10^{10}(\log \log n)^2}$, and
            \item $\gamma_1+\gamma_2+\gamma_3+\gamma_4+\gamma_5 \leq \frac{c_{\eps} \eps^3 n}{10^{14} (\log \log n)^2}$.
        \end{itemize}
        For each block $S_i$ with $s_i \geq 2$, it is easy to check from conditions $2$ and $5$ of the definition of a good graph that at least $1+c_{\eps} s_i/(1000\log \log n)$ vertices are on its border.
        Furthermore, $\cup_{s_i \geq 2} S_i$ contains at least $N/2$ edges by \eqref{eq:BoundSumsi} and, by Lemma \ref{lem:Small-average-degreeSparseRegime}, we have $|\cup_{s_i \geq 2} S_i| \geq N/12$.
        It follows from the above arguments that the number of edges in $E'$ such that not both endpoints are in $V_{rep}$ is at least $\frac{c_{\eps} N}{12000 \log \log n } \geq \frac{c_{\eps} \eps^3 n}{10^{10} \log \log n}$.
        As $|E_{single}(V_{ext},V_{rep})| \leq \frac{n}{10^{10}(\log \log n)^2}$, we get that the number of edges incident to some vertex in $V_{ext}$ in some non-single component in $\mathcal{C}_2$ is at least $\frac{c_{\eps} \eps^3 n}{10^{11} \log \log n}$.
        Letting $M$ be the number of vertices from $V_{ext}$ which belong to some non-single component in $\mathcal{C}_2$, we have, as $\Delta \leq 100 \log \log n$, that 
        \begin{align*}
            M \geq \frac{c_{\eps} \eps^3 n}{10^{13} (\log \log n)^2}.
        \end{align*}
        Plugging this into \eqref{eq:SumNotAloneVext} we obtain
        \begin{align*}
            2(V'-C_1-C_2) \geq M -2(\gamma_1+\gamma_2+\gamma_3+\gamma_4) \geq \frac{c_{\eps} \eps^3 n}{10^{14} (\log \log n)^2},
        \end{align*}
        finishing the proof.
\end{proof}

\section{Proof of Theorem \ref{thm:mainSparseRegime}}
\label{sec:MainProof}

We now proceed to the proof of \Cref{thm:mainSparseRegime}.

\begin{proof}[Proof of \Cref{thm:mainSparseRegime}.]
    By Lemma \ref{lem:Structure2coreSparseRegime}, with high probability, $G$ has a $(n,\eps, \gamma)$-good subgraph $G'$, for some absolute constant $\gamma >0$.
    By \Cref{lem:PropertyWitness}, if there is no witness $S_1 \cup \dots \cup S_k$ of $G'$ (with respect to the embedding associated to $V(G')$) such that there are two distinct indices $i,j$ such that $S_i$ and $S_j$ both contain a vertex of degree at least $3$, then it follows that the set of vertices of degree at least $3$ of $G'$ is reconstructible.
    We are then done, as $G'$ has $\Omega(\eps^3n)$ vertices of degree at least $3$ by definition. \\
    Therefore, it suffices to show that with high probability $G'$ has no such witness, which is also equivalent to proving that after permuting the vertex set of $G$ through a uniformly at random permutation $\sigma$, with high probability the graph $G_{\sigma}'$ has no such witness, where $G'_{\sigma}$ is the graph obtained from $G'$ by relabelling its vertex set according to $\sigma$.
    In other words, it suffices to show that  
    \begin{align}
    \label{eq:MainFirstMoment2SparseRegime}
        \mathcal{Z} :=\sum_k \sum_{\cup_{i \leq k} S_i=V(G')}  \mathbb{P}_{\sigma}(S_1 \cup \dots \cup S_k \text{ is a witness of }G'_{\sigma}) =o_{\eps}(1),
    \end{align} 
    where the second sum is over each connected partition such that at least two of the $S_i$ contain at least one vertex of degree at least $3$, and such that $\sum |E(H[S_i])| \geq |E(H)|/2$.
    Moreover, we remark that if $(S_1, \dots, S_k)$ is a witness and $(S_1, \dots, S_a)$ is a super-block which contains at least one vertex of degree at least $3$, then it is straightforward to see that $(S_1, \dots, S_a, \cup_{a+1 \leq q\leq k} S_q)$ is also a witness.
    Therefore, it is sufficient to show the bound in \eqref{eq:MainFirstMoment2SparseRegime} when the union bound only goes through connected partitions having exactly one super-block $\tilde{S}$, and such that $\tilde{S}$ contains at least one vertex of degree at least $3$. \\
    For the rest of this proof, every time we consider a connected partition $S_1 \cup \dots \cup S_k=V(G')$, we will assume that $S_k$ has maximum size among $S_1, \dots, S_k$. \\
By Lemma \ref{lem:Proba-witness}, we have
\begin{align*}
    \mathcal{Z} \leq Z:=\sum_k \sum_{\cup_{i \leq k} S_i=V(G')} \ind_{B=1} \left(\frac{2}{n} \right)^{V'-C_2-(k-1)},
\end{align*}
where $E'$ is the set of edges whose endpoints are in different $S_i$, $V'$ is the number of vertices being the endpoint of at least one edge in $E'$, $C_2$ is the number of components non-reduced to a single vertex in the graph induced by $E'$, and $B$ is set to $1$ if and only if $S_1 \cup \dots \cup S_k$ has exactly one super-block $\tilde{S}$, and $\tilde{S}$ has at least one vertex of degree at least $3$. \\
We now write $Z=\Sigma_1+\Sigma_2+\Sigma_3$ where $\Sigma_1$ is the contribution where $s_1+\dots+s_{k-1} \leq \log N/ \log \log N$, $\Sigma_2$ the contribution where $\log N/ \log \log N \leq s_1+\dots+s_{k-1} \leq N/2$, and $\Sigma_3$ the contribution where $N/2 \leq s_1+\dots+s_{k-1}$, and $s_i=|S_i|$ for every $i$.
To settle \eqref{eq:MainFirstMoment2SparseRegime}, it suffices to show that $Z =o_{\eps}(1)$. To do this, we successively prove that $\Sigma_1 =o_{\eps}(1)$, $\Sigma_2 =o_{\eps}(1)$ and $\Sigma_3 =o_{\eps}(1)$.
We first prove the bound for $\Sigma_1$. 
Using the second part of Lemma \ref{lem:IneqSuperblock}, we have 
\begin{align}
\label{eq:BoundSigma1SparseRegime}
    \Sigma_1= \sum_k \sum_{\cup_{i \leq k} S_i=V(G')} \ind_{B=1}  \left(\frac{2}{n}\right)^{V'-C_2-(k-1)} \nonumber &\leq \sum_k \sum_{\cup_{i \leq k} S_i=V(G')} \ind_{B=1} \left(\frac{2}{n}\right)^2 \nonumber \\
    &= \sum_k \sum_{s_1+\dots+s_{k-1} \leq \frac{\log N}{\log \log N}} f(s_1,\dots,s_{k-1}) \left(\frac{2}{n}\right)^2,
\end{align}
where we recall that $f(s_1,\dots,s_{k-1})$ is the number of connected partitions $S_1 \cup \dots \cup S_k$ of $V(G')$ such that for each $i \in [k-1]$, we have $|S_i|=s_i$, and the number of super-blocks is exactly 1.
Using the bound given by \Cref{lem:Bound-Connected-Partitions}, we get
\begin{align*}
    \Sigma_1 &\ll \sum_k \sum_{s_1+\dots+s_{k-1} \leq \frac{\log N}{\log \log N}} \frac{(1000 \log \log n)^{2(s_1+\dots+s_{k-1})} }{n} \\
    &\ll \sum_{k \leq \frac{\log N}{\log \log N}} \sum_{a \in [k-1,\frac{\log N}{\log \log N}]} \binom{a+k-1}{k-1}  \frac{(1000 \log \log n)^{2a}}{n} \\
    &\ll \log N \sum_{a \leq \frac{\log N}{\log \log N}} \frac{(10000 \log \log n)^{2a}}{n} \\
    &=o_{\eps}(1).
\end{align*}
We now show that $\Sigma_2 =o_{\eps}(1)$ in the same manner.
Using the first part of Lemma \ref{lem:IneqSuperblock}, we have
\begin{align}
\label{eq:BoundSigma2SparseRegime}
    \Sigma_2 &\leq \sum_k \sum_{\cup_{i \leq k} S_i=V(G')} \ind_{B=1} \left(\frac{2}{n}\right)^{\frac{c_{\eps}(s_1+\dots+s_{k-1})}{400\log \log n}} \nonumber \\
    &\leq \sum_k \sum_{s_1+\dots+s_{k-1} \in [\frac{\log N}{\log \log N},N/2]} f(s_1,\dots,s_{k-1}) \left(\frac{2}{n}\right)^{\frac{c_{\eps}(s_1+\dots+s_{k-1})}{400\log \log n}}.
\end{align}
Using the bound given by \Cref{lem:Bound-Connected-Partitions}, we obtain
\begin{align*}
    \Sigma_2 &\ll \sum_k \sum_{s_1+\dots+s_{k-1} \in [\frac{\log N}{\log \log N},N/2]} (1000 \log \log n)^{2(s_1+\dots+s_{k-1})}\left(\frac{2}{n}\right)^{\frac{c_{\eps}(s_1+\dots+s_{k-1})}{800\log \log n}} \\
    &\ll \sum_{a \in [\frac{\log N}{ \log \log N},N/2]} \sum_{k \leq a} \binom{a+k-1}{k-1} (1000 \log \log n)^{2a}\left(\frac{2}{n}\right)^{\frac{c_{\eps}a}{800\log \log n}} \\
    &\ll \sum_{a \in [\frac{\log N}{ \log \log N},N/2]} (10000 \log \log n)^{2a}\left(\frac{2}{n}\right)^{\frac{c_{\eps}a}{800 \log \log n}} \\
    &=o_{\eps}( 1).
\end{align*}
Finally we show that $\Sigma_3 =o_{\eps}(1)$, still in the same manner. 
By Lemma \ref{lem:CountingEdgesRevealedCase2}, we have
\begin{align}
\label{eq:BoundSigma3SparseRegime}
    \Sigma_3 &\leq \sum_k \sum_{\cup_{i \leq k} S_i=V(G')} \ind_{B=1} \left(\frac{2}{n}\right)^{\frac{\eta N}{(\log \log n)^2}} \leq \sum_k \sum_{s_1+\dots+s_{k-1} \in [N/2,N]} f(s_1,\dots,s_{k-1}) \left(\frac{2}{n}\right)^{\frac{\eta N}{(\log \log n)^2}}.
\end{align}
Using the bound given by \Cref{lem:Bound-Connected-Partitions}, we obtain
\begin{align*}
    \Sigma_3 &\ll \sum_k \sum_{s_1+\dots+s_{k-1} \in [N/2,N]} (1000 \log \log n)^{2(s_1+\dots+s_{k-1})}\left(\frac{2}{n}\right)^{\frac{\eta N}{(\log \log n)^2}} \\
    &\ll \sum_{a \in [N/2,N]} \sum_{k \leq a} \binom{a+k-1}{k-1} (1000 \log \log n)^{2a}\left(\frac{2}{n}\right)^{\frac{\eta N}{(\log \log n)^2}} \\
    &\ll \sum_{a \in [N/2,N]} (10000 \log \log n)^{2a}\left(\frac{2}{n}\right)^{\frac{\eta N}{(\log \log n)^2}}.
\end{align*}
As $N \geq \frac{\eps^3n}{10^5}$, it follows that $\Sigma_3 =o_{\eps}(1)$. This finishes the proof.
\end{proof}

\section{The deterministic setting}
\label{sec:Deterministic}

We start with a proof of \Cref{thm:CounterexampleBT}.

\begin{proof}[Proof of \Cref{thm:CounterexampleBT}]
    Suppose that $n=2^k$ for some integer $k$, and represent each integer $a \leq n-1$ by its binary decomposition $(a_0, \dots, a_{k-1})$.
    We construct the set $V$ as the image of the embedding $f \colon [n] \rightarrow \mathbb{R}$ given by:
    \begin{align*}
        f \colon (a_0,\dots,a_{k-1}) \rightarrow \sum_{i=0}^{k-1} a_i 3^i.
    \end{align*}
    Let $G$ be the graph on vertex set $V$, where two vertices are connected by an edge if and only if their corresponding integers differ in exactly one bit of their binary representation, i.e. $G$ is the hypercube graph of dimension $k$. \\
    For every $0 \leq j \leq k-1$, define another embedding $f_j \colon [n] \rightarrow \mathbb{R}$ as follows:
    \begin{align*}
        f_j \colon (a_0,\dots,a_{k-1}) \rightarrow \sum_{i=0}^{k-1} (1-2 \cdot 1_{i=j}) a_i 3^i.
    \end{align*}
    We remark that for every $j$, the embeddings $f$ and $f_j$ agree on $E(G)$.
    Indeed, if $ab \in E(G)$ where $a$ and $b$ differ in bit $i$, then $|f(a)-f(b)|=|f_j(a)-f_j(b)|=3^i$. 
    
    Let $a,b \leq n-1$ and $j \leq k-1$ be non-negative integers, such that $a$ and $b$ correspond respectively to the sequences of bits $(a_0, \dots, a_{k-1})$ and $(b_0, \dots, b_{k-1})$, with $a_j \neq b_j$ and $ab \notin E(G)$.
    We claim that the following holds.
    \begin{claim}
        We have $|f(a)-f(b)| \neq |f_j(a) -f_j(b)|$.
    \end{claim}
    \begin{proof}
    Suppose by symmetry that $b_j=1$ and $a_j=0$.
    By definition of $f_j$, we have 
        $$f_j(a) -f_j(b)=f(a)-f(b)+2 \cdot 3^j.$$
    It follows that
    \begin{align}
    \label{eq:EmbeddingDiff1}
        f_j(a) -f_j(b) \neq f(a)-f(b).
    \end{align}
    Next, suppose for contradiction that $f(a)-f(b)+2 \cdot 3^j = -(f(a)-f(b))$. 
    Rearranging and simplifying, this gives $f(b)-f(a) =  3^j $.
    Therefore, $\sum_{\ell=0}^{k-1} b_{\ell}3^{\ell} = \sum_{\ell=0}^{k-1} (a_{\ell}+\ind_{\ell=j})3^{\ell}$, and by unicity of the representation in base $3$, we have $b_{\ell}=a_{\ell}$ for every $\ell \in \{0, \dots, k-1 \}$ with $\ell \neq j$. 
    This contradicts the assumption that $ab \notin E(G)$.
    Hence, we have
    \begin{align}
    \label{eq:EmbeddingDiff2}
        f_j(a) -f_j(b) \neq -(f(a)-f(b)).
    \end{align}
    Combining \eqref{eq:EmbeddingDiff1} and \eqref{eq:EmbeddingDiff2}, we obtain the desired result.
    \end{proof}
    From this claim, it is now easy to conclude.
    Indeed, it follows that the distance $ab$ is not reconstructible for every $ab \notin E(G)$. 
    As $G$ does not contain any triangle, the conclusion follows.
\end{proof}

We now give a proof of \Cref{thm:weakBT}.

\begin{proof}[Proof of \Cref{thm:weakBT}]
    We prove this by induction on $n$. 
    For $n=2$, it is trivially true.
    Now suppose that the result is true for every graph on at most $n-1$ vertices.
    Let $G$ be a graph on $n$ vertices and $m$ edges. 
    By \Cref{thm:GaramvolgyiCriteria}, either $G$ is globally rigid in $\mathbb{R}$ (in which case we are done), or $G$ has a NAC-colouring for which $|R_i \cap B_j| \leq 1$ for every $1 \leq i \leq k$ and $1 \leq j \leq \ell$ where $R_1, \dots, R_k$ are $B_1, \dots, B_\ell$ are the vertex sets of the connected components of the subgraph of red and blue edges, respectively. 
    Let $m^r_i$ be the number of red edges in component $R_i$, and let $m^b_j$ be the number of red edges in component $B_j$.
    Let $r_i=|R_i|$ and $b_j=|B_j|$.
    Note that by definition of a NAC-colouring, we have $r_i \leq n-1$ for every $i \in [k]$, and that not all $r_i$ are equal to $1$.
    Similarly, we have $b_j \leq n-1$ for every $j \in [\ell]$, and that not all $b_j$ are equal to $1$.
    Let $\alpha=\frac{m}{n \log n}$. 
    If there exists some $i$ such that $r_i \neq 1$ and $\frac{m^r_i}{r_i \log r_i} \geq \alpha$, then we are done by induction by restricting ourselves to the subgraph $G[R_i]$.
    The same holds if there exists some some $i$ such that $b_j \neq 1$ and $\frac{m^b_j}{b_j \log b_j} \geq \alpha$.
    Therefore, we may assume that $m^r_i \leq \alpha r_i \log r_i$ for every $i \in [k]$, with the strict inequality holding for at least one $i \in [k]$, and similarly that $m^b_j \leq \alpha b_j \log b_j$, for every $j \in [\ell]$, with the strict inequality holding for at least one $j \in [\ell]$.
    Thus we deduce that
    \begin{align}
    \label{eq:m-Entropy}
     m = \sum_i m^r_i + \sum_j m^b_j < \alpha \left( \sum_i r_i \log r_i + \sum_j b_j \log b_j \right).
    \end{align}
    Let $X$ be the set of ordered pairs $(i,j)$ such that $|R_i \cap B_j|=1$.
    Then 
    \begin{align*}
        \sum_i r_i \log r_i + \sum_j b_j \log b_j = \sum_{(i,j) \in X} \log r_i+ \log b_j = \sum_{(i,j) \in X} \log (r_ib_j) = \log \prod_{(i,j) \in X} r_i b_j.
    \end{align*}
    By AM-GM inequality, noting that $|X|=n$, we have that 
    \begin{align*}
        \prod_{(i,j) \in X} r_i b_j \leq \left(\frac{\sum _{(i,j) \in X} r_i b_j}{n}\right)^n \leq \left(\frac{(\sum_i r_i)(\sum_j b_j)}{n}\right)^n =n^n,
    \end{align*}
    Putting all the above together, we obtain that 
    \begin{align*}
        \sum_i r_i \log r_i + \sum_j b_j \log b_j \leq n \log n.
    \end{align*}
    Plugging this in \eqref{eq:m-Entropy} yields $ m < \alpha n \log n=m$, which is a contradiction.
    This finishes the proof.
\end{proof}

To show \Cref{thm:weakBTdense}, we will in fact establish the following stronger result.

\begin{theorem}
\label{thm:Technical-BTdense}
    For every graph $G$ on $n$ vertices and $m$ edges with $n$ sufficiently large, we have for every $\eps \in (0,1/2)$ that $G$ has a globally rigid subgraph in $\mathbb{R}$ of size at least
    \begin{align*}
        \frac{m}{n}2^{-\frac{\log n}{\log (1/\eps)}}-3\frac{\log n}{\eps}.
    \end{align*}
\end{theorem}

From \Cref{thm:Technical-BTdense}, we can easily deduce \Cref{thm:weakBTdense}.

\begin{proof}[Proof of \Cref{thm:weakBTdense} assuming \Cref{thm:Technical-BTdense}]
    We apply \Cref{thm:Technical-BTdense} with $\eps=\frac{\gamma n \log n}{m}$, where $\gamma=6 \cdot 2^{2/\delta}$. 
    Note that $\log (1/\eps) \geq \delta \log n/2$ for sufficiently large $n$, and therefore we get that $G$ contains a globally rigid subgraph of size at least
    \begin{align*}
        \frac{m}{n}2^{-\frac{2\log n}{\delta \log n}}-\frac{3 m}{\gamma n} \geq \frac{m}{n}2^{-2/\delta-1},
    \end{align*}
    as wanted.
\end{proof}

Therefore it suffices to show \Cref{thm:Technical-BTdense}. 
We start by stating the following simple corollary of \Cref{thm:GaramvolgyiCriteria}.

\begin{Corollary}
    \label{lem:Garamvolgyi}
    Suppose that $G$ is not globally rigid in $\mathbb{R}$. Then there exists $k \geq 2$ and a partition of the vertex set of $G$ into non-empty sets $S_1, \dots, S_k$ such that 
    \begin{itemize}
        \item $\sum_{i=1}^k |E(G[S_i])| \geq |E(G)|/2$,
        \item for every vertex $v \in S_i$, and $j \neq i$, we have $|E(v,S_j)| \leq 1$.
    \end{itemize}
\end{Corollary}

We will derive \Cref{thm:Technical-BTdense} from the following main technical result.

\begin{theorem}
\label{thm:DensityIncrem1}
    Let $G=(V,E)$, and suppose that $G$ is not globally rigid in $\mathbb{R}$. 
    Then for every $\eps >0$ we have that
    \begin{itemize}
        \item either there exists $V_1 \subseteq V$ such that $|V_1| \leq \eps |V|$, and if $E_1$ is the set of edges of $G[V_1]$, we have $\frac{|E_1|}{|V_1|} \geq \frac{|E|}{2|V|}$,
        \item or there exists $V_1 \subseteq V$ such that $|V_1| \leq (1-\eps)|V|$ and if $E_1$ is the set of edges of $G[V_1]$, then $\frac{|E_1|}{|V_1|} \geq \frac{|E|}{|V|} -\frac{|V_1|}{|V|}$,
        \item or there exists $V_1 \subseteq V$ such that if $E_1$ is the set of edges of $G[V_1]$, then $\frac{|E_1|}{|V_1|} \geq \frac{|E|}{|V|} - \frac{|V|-|V_1|}{|V|}$.
    \end{itemize}
\end{theorem}

\begin{proof}
    We apply \Cref{lem:Garamvolgyi}, and let $S_1, \dots, S_k$ be the partition of $V(G)$ as given by this corollary. 
    We set $\alpha=\frac{|E|}{|V|}$ for convenience.
    Suppose first that for every $i \in [k]$ we have $|S_i| \leq \eps n$.
    If $|E(S_i)|/|S_i| < \alpha/2$ for every $i$, then 
    \begin{align*}
        |E| \leq 2 \sum_{i=1}^k |E(S_i)| < \alpha  \sum_{i=1}^k |S_i|  = \alpha |V|,
    \end{align*}
    which is a contradiction. This gives outcome 1. \\
    Suppose now that there exists $i \in [k]$ such that $|S_i| \geq \eps |V|$. 
    Rename $A=S_i$ and $B=V \setminus S_i$ for convenience.
    Note that for every vertex $v \in B$, we have $|E(v,A)| \leq 1$.
    Summing over each vertex $v \in B$, we obtain $|E(A,B)| \leq |B|$. 
    Assume for contradiction that we have both $|E(A)|/|A| < \alpha - |B|/|V|$ and $|E(B)|/|B| < \alpha - |B|/|V|$. 
    Then
    \begin{align*}
        |E| = |E(A)| + |E(B)| + |E(A,B)| < \alpha |V|,
    \end{align*}
    which is a contradiction. 
    If $|E(B)|/|B| \geq \alpha - |B|/|V|$, then we immediately obtain outcome $2$.
    If $|E(A)|/|A| \geq \alpha - |B|/|V|$, then we are clearly in outcome $3$.
\end{proof}

We finally show \Cref{thm:Technical-BTdense}.

\begin{proof}[Proof of \Cref{thm:Technical-BTdense}]
    We set $G_0=G$, and we construct $G_{i}$ from $G_{i-1}$ the following way.
    \begin{itemize}
        \item If $G_{i-1}=(V_{i-1}, E_{i-1})$ is globally rigid in $\mathbb{R}$, we terminate the process.
        \item Otherwise, we apply \Cref{thm:DensityIncrem1} to $G_{i-1}$ with $\eps$ to get $V^{*} \subseteq V$ satisfying condition 1, 2 or 3, and we set $G_{i}=G_{i-1}[V^{*}]$.
    \end{itemize}
    Suppose $G_0=(V_0, E_0), \dots, G_t=(V_t, E_t)$ is the construction obtained, and that the process stopped at $G_t$.
    Letting $I_1$, $I_2$ and $I_3$ be respectively the indices $i$ such that the outcome of \Cref{thm:DensityIncrem1} applied to $G_i$ was $1$, $2$ and $3$ respectively.
    Note that we have
    \begin{align}
    \label{eq:Simple-bound-density-increment}
        \frac{|E_t|}{|V_t|} \geq \frac{m}{n}2^{-|I_1|} - |I_2|-\sum_{i=1}^t \frac{|V_{i-1}|-|V_i|}{|V_{i-1}|}
    \end{align}
    Since $1 \leq |V_t| \leq n\eps^{|I_1|} $, we have
    \begin{align}
    \label{eq:BoundI1}
        |I_1| \leq \frac{\log n}{\log (1/\eps)}.
    \end{align}
    Similarly $1 \leq |V_t| \leq n(1-\eps)^{|I_2|} $, and since $\eps < 1/2$, we find that
    \begin{align}
    \label{eq:BoundI2}
        |I_2| \leq \frac{\log n}{\log (1/(1-\eps))} \leq 2\frac{ \log n}{\eps}.
    \end{align}
    Finally, for $n$ sufficiently large, we have
    \begin{align}
    \label{eq:BoundI3}
        \sum_{i=1}^t \frac{|V_{i-1}|-|V_i|}{|V_{i-1}|} \leq \sum_{i=1}^t \sum_{j=|V_i|+1}^{|V_{i-1}|}\frac{1}{j} \leq \sum_{j=1}^{n} \frac{1}{j} \leq 2 \log n. 
    \end{align}
    Plugging \eqref{eq:BoundI1}, \eqref{eq:BoundI2} and \eqref{eq:BoundI3} into \eqref{eq:Simple-bound-density-increment}, we get
    \begin{align}
        |V_t| \geq \frac{|E_t|}{|V_t|} \geq \frac{m}{n}2^{-\frac{\log n}{\log (1/\eps)}} - 2\frac{ \log n}{\eps} - 2 \log n \geq \frac{m}{n}2^{-\frac{\log n}{\log (1/\eps)}} - 3\frac{ \log n}{\eps}.
    \end{align}
    Since $G_t$ is by construction globally rigid in $\mathbb{R}$, this finishes the proof.
\end{proof}

\section{Concluding remarks and further work}
\label{sec:ConcludingRemarks}

As explained in the introduction, Benjamini and Tzalik \cite{benjamini2022determining}, and Girão, Illingworth, Michel, Powierski, and Scott \cite{girao2023reconstructing} studied the reconstruction of an entire point set when distances are revealed according to $G(n,p)$.
Recently, those results were improved by Montgomery, Nenadov, Szabó, and the author \cite{montgomery2024global}.
They showed that a random graph $G$ sampled via the Erd\H{o}s-Rényi evolution becomes globally rigid in $\mathbb{R}$ exactly at the time its minimum degree is $2$, thereby confirming a conjecture posed by Benjamini and Tzalik \cite{benjamini2022determining}.
In other words, whether we fix first the embedding, or generate first the random graph, this does not change the threshold for reconstructibility of the entire point set.
However, the situation is different when the goal is to reconstruct a subset of linear size. 
Indeed, Montgomery, Nenadov, Szabó, and the author showed that for every $\eta >0$, there exists $C>0$, such that, for $G$ sampled as $G(n,C/n)$, with high probability there exists $V' \subseteq V(G)$ such that $|V'| \geq (1-\eta)n$ and $G[V']$ is globally rigid in $\mathbb{R}$.
However, they proved that $1/n$ is not a sharp threshold for the property of having a globally rigid subset in $\mathbb{R}$ of linear size.
More precisely, they showed that there exists $\gamma >0$ such that, for $G$ sampled as $G(n,1.1/n)$, with high probability there does not exist $V' \subseteq V(G)$ such that $G[V']$ is globally rigid in $\mathbb{R}$ and $|V'| \geq \gamma \log n$.
This result shows that the strengthening of \Cref{thm:mainSparseRegime} to global rigidity does not hold, and that the threshold for reconstructing a linear-sized subset of a fixed point set differs from the threshold for the emergence of a globally rigid subgraph of linear size. \\

 Concerning extension to higher dimensions, the question posed by Girão, Illingworth, Michel, Powierski, and Scott \cite{girao2023reconstructing} remains open: what is the threshold at which a linear sized subset of $V$ lying in $\mathbb{R}^d$ can be reconstructed?
 Some progress in this direction was made by Barnes, Petr, Randall Shaw, Sergeev, and the author \cite{barnes2024reconstructing}.
Namely, they proved that $p = \omega(q(n,d))$ with $q(n,d)=n^{-1/\eta(d)+o(1)}$ and $\eta(d)=\frac{d+4}{2}-\frac{1}{d+1}$ is sufficient for reconstructing with high probability a subset of $n-o(n)$ points of a set $V$ consisting of $n$ points in $\mathbb{R}^d$.
However it remains an open problem whether this can be lowered to $p = \omega_d(n^{-1})$. 

\paragraph*{Acknowledgements.} While preparing this manuscript, the author became aware that Let\'{\i}cia Mattos and Tibor Szab\'{o} independently found a counterexample to Benjamini and Tzalik's conjecture similar to \Cref{thm:CounterexampleBT}.
The author would like to thank Julian Sahasrabudhe for many interesting discussions about the subject.
The author is grateful to Carla Groenland for asking whether the methods developed in an earlier version of this article could imply \Cref{thm:weakBTdense}, which was indeed the case.

\bibliographystyle{abbrv}
\renewcommand{\bibname}{Bibliography}
\bibliography{bibliography}

@article{girao2023reconstructing,
  title={Reconstructing a point set from a random subset of its pairwise distances},
  author={Gir{\~a}o, Ant{\'o}nio and Illingworth, Freddie and Michel, Lukas and Powierski, Emil and Scott, Alex},
  journal={SIAM Journal on Discrete Mathematics},
  volume={38},
  number={4},
  pages={2709--2720},
  year={2024},
  publisher={SIAM}
}

@article{garamvolgyi2022global,
  title={Global rigidity of (quasi-) injective frameworks on the line},
  author={Garamv{\"o}lgyi, D{\'a}niel},
  journal={Discrete Mathematics},
  volume={345},
  number={2},
  pages={112687},
  year={2022},
  publisher={Elsevier}
}

@article{ding2014anatomy,
  title={Anatomy of the giant component: The strictly supercritical regime},
  author={Ding, Jian and Lubetzky, Eyal and Peres, Yuval},
  journal={European Journal of Combinatorics},
  volume={35},
  pages={155--168},
  year={2014},
  publisher={Elsevier}
}

@article{barnes2024reconstructing,
  title={Reconstructing almost all of a point set in $\mathbb{R}^{d}$ from randomly revealed pairwise distances},
  author={Barnes, Douglas and Petr, Jan and Portier, Julien and Shaw, Benedict Randall and Sergeev, Alan},
  journal={arXiv preprint arXiv:2401.01882},
  year={2024}
}

@article{benjamini2022determining,
  title={Determining a points configuration on the line from a subset of the pairwise distances},
  author={Benjamini, Itai and Tzalik, Elad},
  journal={arXiv preprint arXiv:2208.13855},
  year={2022}
}

@article{bollobas1988isoperimetric,
  title={The isoperimetric number of random regular graphs},
  author={Bollob{\'a}s, B{\'e}la},
  journal={European Journal of combinatorics},
  volume={9},
  number={3},
  pages={241--244},
  year={1988},
  publisher={Academic Press Ltd. London, UK, UK}
}

@book{bollobas2006art,
  title={The art of mathematics: Coffee time in Memphis},
  author={Bollob{\'a}s, B{\'e}la},
  year={2006},
  publisher={Cambridge University Press}
}

@article{benjamini2014mixing,
  title={The mixing time of the giant component of a random graph},
  author={Benjamini, Itai and Kozma, Gady and Wormald, Nicholas},
  journal={Random Structures \& Algorithms},
  volume={45},
  number={3},
  pages={383--407},
  year={2014},
  publisher={Wiley Online Library}
}

@article{chvatal1991almost,
  title={Almost all graphs with 1.44 n edges are 3-colorable},
  author={Chv{\'a}tal, Vasek},
  journal={Random Structures \& Algorithms},
  volume={2},
  number={1},
  pages={11--28},
  year={1991},
  publisher={Wiley Online Library}
}

@book{janson2011random,
  title={Random graphs},
  author={Janson, Svante and Luczak, Tomasz and Rucinski, Andrzej},
  year={2011},
  publisher={John Wiley \& Sons}
}

@article{connelly2005generic,
  title={Generic global rigidity},
  author={Connelly, Robert},
  journal={Discrete \& Computational Geometry},
  volume={33},
  pages={549--563},
  year={2005},
  publisher={Springer}
}

@article{gortler_characterizing_2012,
  title={Characterizing generic global rigidity},
  author={Gortler, Steven J. and Healy, Alexander D. and Thurston, Dylan P.},
  journal={American Journal of Mathematics},
  volume={132},
  number={4},
  pages={897--939},
  year={2010},
  publisher={Johns Hopkins University Press}
}

@article{lew23randomrigid,
  title={Sharp threshold for rigidity of random graphs},
  author={Lew, Alan and Nevo, Eran and Peled, Yuval and Raz, Orit E.},
  journal={Bulletin of the London Mathematical Society},
  volume={55},
  number={1},
  pages={490--501},
  year={2023},
  publisher={Wiley Online Library}
}

@incollection{jordan2017global,
  title={Global rigidity},
  author={Jord{\'a}n, Tibor and Whiteley, Walter},
  booktitle={Handbook of Discrete and Computational Geometry},
  pages={1661--1694},
  year={2017},
  publisher={Chapman and Hall/CRC}
}

@article{raz2023dense,
  title={Dense graphs have rigid parts},
  author={Raz, Orit E and Solymosi, J{\'o}zsef},
  journal={Discrete \& Computational Geometry},
  volume={69},
  number={4},
  pages={1079--1094},
  year={2023},
  publisher={Springer}
}

@article{montgomery2024global,
  title={Global rigidity of random graphs in $\mathbb{R}$},
  author={Montgomery, Richard and Nenadov, Rajko and Portier, Julien and Szab{\'o}, Tibor},
  journal={arXiv preprint arXiv:2401.10803},
  year={2024}
}

@article{peled2024rigidity,
  title={On the Rigidity of Random Graphs in high-dimensional spaces},
  author={Peled, Yuval and Peleg, Niv},
  journal={arXiv preprint arXiv:2412.13127},
  year={2024}
}

@article{krivelevich2023rigid,
  title={Rigid partitions: from high connectivity to random graphs},
  author={Krivelevich, Michael and Lew, Alan and Michaeli, Peleg},
  journal={Journal of Combinatorial Theory, Series B},
  volume={175},
  pages={126--170},
  year={2025},
  publisher={Elsevier}
}

@article{villanyi2025every,
  title={Every $d(d+1)$-connected graph is globally rigid in $\mathbb{R}^d$},
  author={Vill{\'a}nyi, Soma},
  journal={Journal of Combinatorial Theory, Series B},
  volume={173},
  pages={1--13},
  year={2025},
  publisher={Elsevier}
}

\end{document}